\documentclass[11pt]{amsart}

\usepackage[backend=biber]{biblatex}
\usepackage{graphicx}%
\usepackage{multirow}%
\usepackage{amsmath,amssymb,amsfonts}%
\usepackage{amsthm}%
\usepackage{mathrsfs}%
\usepackage[title]{appendix}%
\usepackage{xcolor}%
\usepackage{textcomp}%
\usepackage{manyfoot}%
\usepackage{booktabs}%
\usepackage{algorithm}%
\usepackage{algorithmicx}%
\usepackage{algpseudocode}%
\usepackage{listings}%
\usepackage{mathtools}
\usepackage{relsize}
\usepackage{setspace}
\usepackage{enumitem}
\usepackage{tikz-cd}
\usepackage[all]{xy}

\newtheorem{theorem}{Theorem}[section]
\newtheorem{proposition}[theorem]{Proposition}
\newtheorem{lemma}[theorem]{Lemma}
\newtheorem{corollary}[theorem]{Corollary}

\theoremstyle{definition}
\newtheorem{definition}[theorem]{Definition}

\theoremstyle{remark}
\newtheorem{remark}[theorem]{Remark}
\newtheorem{example}[theorem]{Example}
\AtBeginEnvironment{example}{%
  \pushQED{\qed}
}
\AtEndEnvironment{example}{\popQED\endexample}

\newcommand{\Z}{\ensuremath{\mathbb{Z}}}

\newcommand{\F}{\ensuremath{\mathbb{F}}}

\newcommand{\comp}{\mathsf{c}}

\newcommand{\zmap}{z} 
\newcommand{\vmap}{v} 

\newcommand{\set}[2]{\left\{ \, {#1} \; \middle\vert \; {#2} \, \right\}}

\newcommand{\deff}[1]{\textbf{\emph{#1}}}
\newcommand{\abs}[1]{\lvert {#1} \rvert}
\newcommand{\hyph}{\mbox{-}}
\newcommand{\symdif}{\mathbin{\bigtriangleup}}
\newcommand{\bigsymdif}{\mathop{\mathlarger{\mathlarger{\mathlarger{\bigtriangleup}}}}\displaylimits}
\newcommand{\func}[3]{{#1} \colon {#2} \to {#3}}

\newcommand{\pow}[1]{\mathcal{P}(#1)}
\newcommand{\vv}[1]{\boldsymbol{#1}}

\DeclareMathOperator{\aff}{aff}
\DeclareMathOperator{\spann}{span}

\DeclareMathOperator{\AG}{AG}

\DeclareMathOperator{\venn}{Venn}

\usepackage{parskip}

\title[Affine Equivalence of Subsets of $\mathbb{F}_2^n$ via Venn Diagrams]{Affine Equivalence of Subsets of $\mathbb{F}_2^n$ via Venn Diagrams and Applications to Sidon Sets}

\author[Calta]{Kariane Calta}
\address{Department of Mathematics \& Statistics, Vassar College, Poughkeepsie, New York 12604}
\email{kacalta@vassar.edu}

\author[Covey]{Sarah Covey}
\address{Department of Mathematics, Harvey Mudd College, Claremont, CA 91711}
\email{hcovey@hmc.edu}

\author[Goldberg]{Timothy E.~Goldberg}
\address{Department of Mathematics, Lenoir-Rhyne University, Hickory, North Carolina 28601}
\email{timothy.goldberg@lr.edu}

\author[Rose]{Lauren L.~Rose}
\address{Department of Mathematics, Bard College, Annandale-on-Hudson 12504}
\email{rose@bard.edu}

\author[Rose-Levine]{Daniel Rose-Levine}
\address{Department of Mathematics, Bard College, Annandale-on-Hudson 12504}
\email{dr6048@bard.edu}

\date{\today}

\begin{document}

\begin{abstract} 
    Two subsets $S$ and $T$ of $\F_2^n$ are \emph{affinely equivalent} if there is an affine automorphism of $\F_2^n$ taking $S$ to $T$. Given a basis of the affine span of $S$, we can construct a Venn diagram whose regions partition $S$. We prove that any two bases of $\aff(S)$ will have the same Venn diagram up to a linear permutation of the Venn regions. Moreover, we prove that two sets are affinely equivalent if and only if there is a cardinality-preserving linear permutation from the Venn regions of $S$ to the Venn regions of $T$. We use these results to classify certain Sidon sets up to affine equivalence.
\end{abstract}

\maketitle

\setcounter{tocdepth}{1}
\tableofcontents

\section{Introduction}

The work described in this paper was inspired by studying the card game \emph{Quads}, created by Rose and Pereira as \emph{SuperSET} in 2013, \cite{Rosepereira2013Super}, and published by the Association for Women in Mathematics as \emph{EvenQuads}, \cite{evenquads}.
By coding the 64 cards as binary vectors in $\F_2^6$, the game can be viewed as a model for $\operatorname{AG}(6,2)$, the finite affine geometry of order $2$ and dimension $6$. This can be expanded to a game with $2^n$ cards for any positive $n$, which corresponds to $\F_2^n$ and models $\operatorname{AG}(n,2)$. The goal of the game is to identify four cards fitting a particular pattern, which corresponds to four vectors $\vv{a}, \vv{b}, \vv{c}, \vv{d} \in \F_2^n$ such that $\vv{a} + \vv{b} + \vv{c} + \vv{d} = \vv{0}$.  A set of four such cards or vectors is called a \emph{quad}, and the quads are exactly the affine planes in $\F_2^n$. For a detailed description, see \cite{lamatpaper}.

Of particular interest are quad-free collections of cards (or vectors), called \emph{quad caps}, \emph{Sidon sets}, or \emph{2-caps} (\cite{Bennett}). Sidon sets have many applications, for example in cryptography for designing secure systems, in coding theory for constructing error-correcting codes, and in signal processing for creating efficient frequency-hopping sequences. See for example, \cite{leon}, \cite{osorio}, and \cite{niu}. An active research area in cryptography and coding theory is the study and construction of maximal Sidon sets in $\F_2^n$, and in particular determining the maximum size of a Sidon set in each dimension. This is the analog of the famous \emph{Cap Set problem} for the card game SET,
which remains unsolved in dimensions $n \geq 7$. Although asymptotic bounds are known for maximum size Sidon sets, \cite{taitwon}, explicit constructions of all possible caps are only known for $n \leq 7$, \cite{quad128}, and exact maximum sizes are unknown for dimensions $n > 10$. A subsidiary question is to determine when two Sidon sets of the same size are affinely equivalent. Recent papers on this and related topics include \cite{potts}, \cite{nagy}, \cite{RedmanRoseWalker}, \cite{carlet}, and \cite{daniel-sproj}.

For the remainder of this paper, we will refer to Sidon sets in $\F_2^n$ as \emph{quad caps} or just \emph{caps}.
Although each quad cap has a well-defined cardinality and dimension, these two numbers do not characterize the set, even up to affine equivalence. Quad caps generally do not have any apparent algebraic or geometric structure, and so determining when they are affinely equivalent can be challenging. In this paper, we develop a \emph{Venn diagram method} for determining whether any two subsets of $\F_2^n$, not just quad caps, are affinely equivalent. 

In order to do this, we make use of the fact that the power set $\pow{S}$ of a set $S$ is a vector space over $\F_2$ under the operation of symmetric difference of sets. In Section~\ref{sec:affine}, we recall the key ideas from affine linear algebra over $\Z_2$, and define the \emph{even zero-sum sets} of $S \subseteq \F_2^n$, which form a linear subspace $E(S)$ of $\pow{S}$. This subspace captures the affine geometric structure of $S$, and a bijection between subsets of $\F_2^n$ is an affine equivalence if and only if it induces a bijection of the even zero-sums sets, (Theorem~\ref{key_thm}).

In Section~\ref{sec:venn-regions}, we examine linear subspaces $\mathcal{V} \subseteq \pow{S}$ in general, and define a partition of $S$ into \emph{Venn regions}, which are the connected components of the Venn diagram of sets in $\mathcal{V}$. The collection $\venn(S, \mathcal{V})$ of these Venn regions is also a vector space over $\F_2$ which encodes information about $\mathcal{W}$. In Section~\ref{sec:venn-maps}, we define isomorphisms between pairs $(S, \mathcal{V})$ and $(T, \mathcal{W})$, where $S, T \subseteq \F_2^n$ and $\mathcal{V}, \mathcal{W} \subseteq \pow{S}$ are linear subspaces. In particular, we show that these pairs are isomorphic if and only if there is a cardinality-preserving, linear bijection $\venn(S, \mathcal{V}) \to \venn(T, \mathcal{W})$.  

In Section~\ref{sec:zero-sums}, we study $E(S)$ in more depth, and apply the results of previous sections to obtain our main theorem, that subsets $S, T \subseteq \F_2^n$ are affinely equivalent if and only if there is a cardinality-preserving, linear bijection $\venn(S, E(S)) \to \venn(T, E(T))$.

In the remaining sections of the paper, we apply the tools developed in Sections 3--5 to study quad caps. In Section~\ref{sec:diff3} we obtain a complete characterization of caps of size $k$ and dimension $d=k-3$. In Section~\ref{sec:dim7caps}, we consider the templates introduced in \cite{quad128} for $7$-dimensional caps of sizes 10--12, and verify which templates are affinely equivalent using the Venn diagram method.

\section{Affine linear algebra in $\F_2^n$} \label{sec:affine}

The set $\AG(n,q)$ is the finite affine geometry of dimension $n$ and order $q$, where $q$ is a power of a prime. A standard model for $\AG(n,q)$ consists of vectors in $\F_q^n$, where $\F_q$ is the field of order $q$. In this model the affine subspaces (called \deff{flats}) are just translates of linear subspaces. A detailed description of the geometry of $\AG(n,q)$, and in particular $\AG(n,2)$, can be found in \cite{lamatpaper}. In particular, we assume the reader is familiar with the definitions and basic properties of affine combination, affine independence and dependence, affine span, affine subspace (also called a \deff{flat}), affine basis and dimension, affine map, and affine isomorphism and automorphism.

\begin{definition}
    Let $V$ be a vector space over a field $\F$, and let $S \subseteq V$. A \deff{basis} for $S$ is a basis for the flat $\aff(S)$, and the dimension of $S$ is $\dim(S) = \dim \aff(S)$.
    Subsets $S, T \subseteq V$ are \deff{affinely equivalent} if there is an affine automorphism $f$ of $V$ such that $f(S) = T$, or equivalently if there is a bijection $\func{g}{S}{T}$ that can be extended to an affine automorphism of $V$. 
\end{definition}

For vector spaces over the field $\F_2 = \{0, 1\}$, the standard definitions take on a somewhat peculiar character. In particular, the affine combinations are precisely sums of an odd number of elements, and so $S \subseteq \F_2^n$ is affinely dependent if and only if there is a nonempty even subset of $S$ whose elements sum to $\vv{0}$. Moroever, if $F \subseteq \F_2^n$ is a flat of dimension $d$, then $\abs{F} = 2^{d+1}$.


        




        



Inspired by this characterization of affinely dependent sets, we make the following definition.

\begin{definition} \label{def:zero-sums}
A subset $\{\vv{x}_1, \ldots, \vv{x}_{\ell}\}$ of $\F_2^n$ is a \deff{zero-sum set} if
\[
    \vv{x}_1 + \cdots + \vv{x}_l = \vv{0}.
\]
An \deff{even zero-sum set} is a zero-sum set of even cardinality. We define the empty set $\varnothing$ to be an even zero-sum set. We denote the collection of even zero-sum subsets of a set $S \subseteq \F_2^n$ by $E(S)$.
\end{definition}

Note that we are not concerned with odd zero-sum sets, because those do not pertain to the affine geometric structure of $\F_2^n$.

\begin{remark}
    Throughout the remainder of this paper, we will regularly write a subset $S \subseteq \F_2^n$ as the disjoint union $S = B \sqcup D$, where $B$ is a basis for $S$ and $D = S - B$.
\end{remark}

\begin{example} \label{running-example-1}
    Suppose $n \geq 9$, and consider the subset
    \[
        S
        = B \sqcup D
        = \{\vv{a}_1, \ldots, \vv{a}_{10}\} \sqcup \{\vv{w}_1, \vv{w}_2, \vv{w}_3 \}
        \subseteq \F_2^n
    \]
    where $B$ is affinely independent and
    \begin{align*}
        \vv{w}_1 &= \vv{a}_1 + \cdots + \vv{a}_7, \\
        \vv{w}_2 &= \vv{a}_4 + \cdots + \vv{a}_{10}, \\
        \vv{w}_3 &= \vv{a}_3 + \vv{a}_4 + \vv{a}_8.
    \end{align*}
    Hence $B$ is a basis for $S$, and $\abs{S} = 13$ and $\dim(S) = 9$. Since $\vv{w}_1 + \vv{a}_1 + \cdots + \vv{a}_7 = \vv{0}$, the set $X_1 = \{ \vv{w}_1, \vv{a}_1, \ldots, \vv{a}_7 \}$ is an element of $E(S)$. Similarly, from the equations for $\vv{w}_2$ and $\vv{w}_3$ we obtain:
    \begin{align*}
        X_2 &= \{ \vv{w}_2, \vv{a}_4, \ldots, \vv{a}_{10} \} \in E(S), \\
        X_3 &= \{ \vv{w}_3, \vv{a}_3, \vv{a}_4, \vv{a}_8 \} \in E(S).
    \end{align*}
    Another element of $E(S)$ is derived by adding the expressions above for $\vv{w}_1$ and $\vv{w}_2$. Since $\vv{w}_1 + \vv{w}_2 = \vv{a}_1 + \vv{a}_2 + \vv{a}_3 + \vv{a}_8 + \vv{a}_9 + \vv{a}_{10}$, we have
    \[
        \{ \vv{w}_1, \vv{w}_2, \vv{a}_1, \vv{a}_2, \vv{a}_3, \vv{a}_8, \vv{a}_9, \vv{a}_{10} \} \in E(S).
    \]
    We will prove in Section~\ref{sec:zero-sums} that all elements of $E(S)$ are obtained by combining the expressions for $\vv{w}_1$, $\vv{w}_2$, and $\vv{w}_3$ in this way.

    For a more concrete version of this example, one could choose explicit vectors for $B$, such as the standard basis for $\F_2^9$ along with the zero vector, and then compute the other vectors. For convenience, we write these vectors as bitstings.
    \begin{align*}
        \vv{a}_1 &= 100 \, 000 \, 000 & \vv{a}_8 &= 000 \, 000 \, 010 \\
        \vv{a}_2 &= 010 \, 000 \, 000 & \vv{a}_9 &= 000 \, 000 \, 001 \\
        \vv{a}_3 &= 001 \, 000 \, 000 & \vv{a}_{10} &= 000 \, 000 \, 000 \\
        \vv{a}_4 &= 000 \, 100 \, 000 & \\
        \vv{a}_5 &= 000 \, 010 \, 000 & \vv{w}_1 &= 111 \, 111 \, 100 \\
        \vv{a}_6 &= 000 \, 001 \, 000 & \vv{w}_2 &= 000 \, 111 \, 111 \\
        \vv{a}_7 &= 000 \, 000 \, 100 & \vv{w}_3 &= 001 \, 100 \, 010
    \end{align*}
\end{example}

Recall that any map $\func{f}{X}{Y}$ between sets induces a map $\func{f_{\ast}}{\pow{X}}{\pow{Y}}$.

\begin{proposition} \label{prop:preserves_even_zero-sums}
    A function $\func{f}{\F_2^n}{\F_2^m}$ is affine if and only if $f_{\ast}(E(\F_2^n)) \subseteq E(\F_2^m)$.
\end{proposition}

\begin{proof}
    Let $\func{f}{\F_2^n}{\F_2^m}$. Note that $f_{\ast}(\varnothing) = \varnothing \in E(\F_2^m)$. By definition $f$ is affine if and only if it preserves affine combinations. Recall that the affine combinations in $\F_2^n$ are precisely the sums of an odd number of points.

    Let $\vv{x}_1, \ldots, \vv{x}_{2\ell+1} \in  \F_2^n$ be distinct, and let $\vv{y} = \vv{x}_1 + \cdots + \vv{x}_{2\ell+1}$. Observe that $X = \{\vv{y}, \vv{x}_1, \ldots, \vv{x}_{2\ell+1}\} \in E(\F_2^n)$.
    Then
    \[
        \begin{array}{l}
            f(\vv{x}_1 + \cdots + \vv{x}_{2\ell+1})
            = f(\vv{x}_1) + \cdots + f(\vv{x}_{2\ell+1}) \\[1ex]
            \phantom{mmm} \iff
            f(\vv{y})
            = f(\vv{x}_1) + \cdots + f(\vv{x}_{2\ell+1}) \\[1ex]
            \phantom{mmm} \iff
            f_{\ast} (X) = \left\{ f(\vv{y}), f(\vv{x}_1), \ldots, f(\vv{x}_{2\ell+1}) \right\} \in E(\F_2^m).
        \end{array}
    \]
    Since every nonempty element of $E(\F_2^n)$ corresponds to an affine combination, this proves the proposition.
\end{proof}

\begin{theorem} \label{thm:preserve-zero-sums}
    Two sets $S, T \subseteq \F_2^n$ are affinely equivalent if and only if there is a bijection $\func{g}{S}{T}$ such that $g_{\ast}|_{E(S)}$ is a bijection onto $E(T)$.
\end{theorem}

\begin{proof}
    Let $S, T  \subseteq \F_2^n$.

    Suppose $S$ and $T$ are affinely equivalent. Then there is an affine automorphism $f$ of $\F_2^n$ such that $f(S) = T$. Define $\func{g}{S}{T}$ by setting $g = f|_{S}$. Since $f$ is injective, we know $g$ is also injective, and because $g(S) = f(S) = T$ we also know $g$ is surjective. Hence $g$ is bijective. Note that $g^{-1} = (f^{-1})|_{T}$.

    By Proposition~\ref{prop:preserves_even_zero-sums}, we know that $f$ and $f^{-1}$ preserve even zero-sums, and hence so do $g$ and $g^{-1}$. Therefore $g_{\ast}|_{E(S)}$ is bijective onto $E(T)$.

    \bigskip

    Suppose there is a bijective function $\func{g}{S}{T}$ such that $g_{\ast}|_{E(S)}$ is a bijection onto $E(T)$. Let $B = \{\vv{x}_0, \ldots, \vv{x}_d\} \subseteq S$ be a basis for $S$. We will show that $g_{\ast}(B) \subseteq T$ is affinely independent.
        
    Let $X$ be an even zero-sum subset of $g_{\ast}(B)$. By assumption we have $g_{\ast}^{-1}(X) \in E(S)$, so $g_{\ast}^{-1}(X)$ is an even zero-sum set in the independent set $B$, so we know that $g_{\ast}^{-1}(X) = \emptyset$. Since $g$ is bijective, this implies that $X = \emptyset$. Therefore $g_{\ast}(B)$ contains no nonempty even zero-sum sets, so by the same theorem just cited, it is affinely independent.
 
    Let $\vv{y}_i = g(\vv{x}_i)$ for $0 \leq i \leq d$. The affinely independent sets $B$ and $g_{\ast}(B)$ can each be extended to affine bases
    \[
        \{ \vv{x}_0, \ldots, \vv{x}_d, \ldots, \vv{x}_n \}
        \quad \text{and} \quad
        \{ \vv{y}_0, \ldots, \vv{y}_d, \ldots, \vv{y}_n \}
    \]
    for $\F_2^n$.
    Then the function $\func{f}{\F_2^n}{\F_2^n}$ defined by setting $f(\vv{x}_i) = \vv{y}_i$ for $0 \leq i \leq n$ and extending affine linearly is an affine automorphism.

    We now show that $f(S) = T$, which will complete the proof that $S$ and $T$ are affinely equivalent. Let $\vv{x} \in S$. Since $B$ is a basis for $S$, there are $\vv{x}_{i_1}, \ldots, \vv{x}_{i_{2j+1}} \in B$ whose sum is $\vv{x}$. Then
    \[
        X = \{\vv{x}, \vv{x}_{i_1}, \ldots, \vv{x}_{i_{2j+1}} \} \in E(S).
    \]
    By assumption
    \[
        g_{\ast}(X) = \{g(\vv{x}), g(\vv{x}_{i_1}), \ldots, g(\vv{x}_{i_{2j+1}}) \} \in E(T),
    \]
    so
    \begin{align*}
        g(\vv{x})
        &= g(\vv{x}_{i_1}) + \cdots + g(\vv{x}_{i_{2j+1}}) \\
        &= f(\vv{x}_{i_1}) + \cdots + f(\vv{x}_{i_{2j+1}}) & (\text{since $f|_B = g|_B$}) \\
        &= f(\vv{x}_{i_1} + \cdots + \vv{x}_{i_{2j+1}}) & (\text{since $f$ is affine}) \\
        &= f(\vv{s}).
    \end{align*}
    Hence $f|_S = g|_S$, so $f(S) = g(S) = T$.

\end{proof}




\section{Venn diagrams and regions} \label{sec:venn-regions}

Let $S$ be any set. It is well-known that the power set $\pow{S}$ of $S$ has the structure of a \deff{Boolean ring} with addition given by the \deff{symmetric difference} $\symdif$, defined by 
\[
    X \symdif Y
    = (X \cup Y) - (X \cap Y)
\]
for $X, Y \in \pow{S}$, and multiplication given by set intersection. The zero element is the empty set and the unity element is $S$. The symmetric difference is both commutative and associative, and intersection distributes over symmetric difference.

The following two lemmas are standard results about $\symdif$.

\begin{lemma} \label{lem:sym-dif}
    Let $S$ be a set, and let $X_1, \ldots, X_r$ be a finite collection of subsets of $S$. For any $x \in S$, $x \in X_1 \symdif \cdots \symdif X_r$ if and only if $x \in X_i$ for exactly an odd number of $i$, $1 \leq i \leq r$.
\end{lemma}

\begin{lemma} \label{symdif-preserved}
    Let $\func{f}{S}{T}$ be an injective function between sets $S$ and $T$. Then $f(A \symdif B) = f(A) \symdif f(B)$ for all $A, B \subseteq S$.
\end{lemma}


In this paper, we focus on $\pow{S}$ as a vector space over $\F_2$.

\begin{proposition} \label{prop:bases}
    Let $S$ be a nonempty set of size $k$. Then the singleton subsets of $S$ form a basis for $\pow{S}$ and hence $\dim \pow{S} = k$.
\end{proposition}

\begin{proof} \leavevmode
    The singleton subsets of $S$ span $\pow{S}$, because for any $X \in \pow{S}$ we have
    \[
        X = \bigsymdif_{x \in X} \{x\}.
    \]
    Also, for any $x_1, \ldots, x_r \in S$ we have
    \[
        \bigsymdif_{i=1}^r \{x_i\}
        = \{x_1, \ldots, x_r\} \neq \varnothing,
    \]
    so the singleton subsets are linearly independent.
    Thus the singleton subsets form a basis for $\pow{S}$, so $\dim \pow{S} = |S| = k$.     
\end{proof}

\begin{definition}
    Let $S$ be a finite set and let $\mathcal{V} = \{X_1, \ldots, X_r\} \subseteq \pow{S}$. The \deff{Venn regions} of $S$ with respect to $\mathcal{V}$ are the sets
    \[
        R_I = \left( \bigcap_{i \in I} X_i \right) \cap \left( \bigcap_{i \notin I} X_i^{\comp} \right)
    \]
    for $I \subseteq \{1, \ldots, r\}$, where $X_i^{\comp} = S - X_i$.
    The set
    \[
        R_{\varnothing} = X_1^{\comp} \cap \cdots \cap X_r^{\comp}
    \]
    is the \deff{isolated point set} for $\mathcal{V}$.

    If $x, y \in S$ are contained in the same Venn region, then we write $x \sim_{\mathcal{V}} y$.
\end{definition}

\begin{remark}
    Observe that the Venn regions defined above are exactly the connected components of a Venn diagram for the sets $X_1, \ldots, X_r$. The isolated point set is the region outside of all of these sets, $X_1^{\comp} \cap \cdots X_r^{\comp}$.
    The nonempty Venn regions are pairwise disjoint and partition $S$, and thus $\sim_{\mathcal{V}}$ is an equivalence relation.
    Hence the quotient set $S / {\sim_{\mathcal{V}}}$ is the set of nonempty Venn regions.
    Note also that $x \sim_{\mathcal{V}} y$ if and only if $x$ and $y$ are contained in exactly the same elements of $\mathcal{V}$.
\end{remark}

In the following lemma, we show that each set in $\mathcal{V}$ is a union of Venn regions. Later, in Proposition~\ref{prop:intersect-and-union}, we will also prove that if $\mathcal{V}$ is a linear subspace, then each Venn region is an intersection of sets in $\mathcal{V}$.

\begin{lemma} \label{lem:venn-containment}
    Let $S$ be a finite set and let $\mathcal{V} = \{X_1, \ldots, X_r\} \subseteq \pow{S}$.
    \begin{enumerate}
        \item For any $X_i \in \mathcal{V}$ and Venn region $R$ from $\mathcal{V}$, either $R \cap X_i = \varnothing$ or $R \subseteq X_i$.
        \item Each $X_i \in \mathcal{V}$ is a union of Venn regions from $\mathcal{V}$.
    \end{enumerate}
\end{lemma}

\begin{proof}
    Let $X_i \in \mathcal{V}$ and let $R$ be a Venn region from $\mathcal{V}$. Then either $R \cap X_i = \varnothing$ or $R \cap X_i \neq \varnothing$. Assume the latter, so there exists $\vv{x} \in R \cap X_i$. Then $\vv{x} \in R$ and $\vv{x} \in X_i$. Then for any $\vv{r} \in R$ we have $\vv{x} \sim_{\mathcal{V}} \vv{r}$, so $\vv{x}$ and $\vv{r}$ are contained in exactly the same elements of $\mathcal{V}$. Since $\vv{x} \in X_i$, it follows that $\vv{r} \in X_i$. Therefore $R \subseteq X_i$.

    That each $X_i \in \mathcal{V}$ is a union of Venn regions follows from the facts that each $\vv{x} \in X_i$ is in some Venn region, and if a Venn region intersects $X_i$ at all, then it must be contained in $X_i$.
\end{proof}

Since each Venn region is determined by a subset $I \subseteq \{1, \ldots, r\}$, we can parametrize the Venn regions using coordinates in $\F_2^r$. The Venn region $R_I$ corresponds to the vector $\vv{a}_I \in \F_2^r$ defined by
\[
    \vv{a}_{I} = \sum_{i \in I} \vv{e}_i,
\]
where $\vv{e}_1, \ldots, \vv{e}_r$ are the standard basis vectors for $\F_2^r$. The set $I$ is called the \deff{support} of the vector $\vv{a}_I$. The following is an easy consequence of the definitions.

We display the vector labels for the Venn regions of Venn diagrams of two, three, and four sets in Figure~\ref{fig:venn-labels}. For convenience, we write these vectors as bitstrings.

\begin{figure}
    \centering

    \begin{tikzpicture}[set/.style={fill=cyan,fill opacity=0.1}, scale = 0.4]
    \def\firstcircle{(0:2cm) circle (3cm)}
    \def\secondcircle{(180:2cm) circle (3cm)}
        
    \draw (-5.3,-4) rectangle (5.3,4.3);
    
    \draw[set] \firstcircle;
    \draw[set] \secondcircle;

    \draw node at (4.4cm,3.4cm) {$X_2$};
    \draw node at (-4.4cm,3.4cm) {$X_1$}; 
        
        \draw node at (90:3.2cm) {$00$};    
        \draw node at (0:0cm) {$11$};
        \draw node at (180:2.7cm) {$10$};
        \draw node at (0:2.7cm) {$01$};
    \end{tikzpicture}
    \begin{tikzpicture}[set/.style={fill=cyan,fill opacity=0.1}, scale = 0.6]
        \def\firstcircle{(90:2cm) circle (3cm)}
        \def\secondcircle{(210:2cm) circle (3cm)}
        \def\thirdcircle{(330:2cm) circle (3cm)}

        \draw (-6.2cm,-4.75cm) rectangle (6.2cm,5.75cm);
        \draw[set] \firstcircle;
        \draw[set] \secondcircle;
        \draw[set] \thirdcircle;

        \draw node at (90:5.4cm) {$X_1$};    
        \draw node at (210:5.4cm) {$X_2$};    
        \draw node at (330:5.4cm) {$X_3$};    
        
            \draw node at (-5cm,4.5cm) {$000$};    
            \draw node at (90:2.7cm) {$100$};    
            \draw node at (210:2.7cm) {$010$};
            \draw node at (330:2.7cm) {$001$};
            \draw node at (0,0) {$111$};
            \draw node at (270:2cm) {$011$};
            \draw node at (30:2cm) {$101$};
            \draw node at (150:2cm) {$110$};
    \end{tikzpicture}

    \begin{tikzpicture}[set/.style={fill=cyan,fill opacity=0.1}, scale=2]
 
    \draw[set,
        rotate =45] (0,0) ellipse (2cm and 1cm);
 
    \draw[set,
        rotate =-45] (0,0) ellipse (2cm and 1cm);
 
    \draw[set,
        xshift=1cm,
        yshift=-0.705cm,
        rotate =45] (0,0) ellipse (2cm and 1cm);
 
    \draw[set,
        xshift=-1cm,
        yshift=-0.705cm,
        rotate =-45,] (0,0) ellipse (2cm and 1cm);

    \draw node at (163:2.75) {$X_1$};
    \draw node at (135:2.2) {$X_2$};
    \draw node at (45:2.2) {$X_3$};
    \draw node at (17:2.75) {$X_4$};
 
    \draw node at (90:1.75) {$0000$};
    \draw node at (90:0.5) {$0110$};
    \draw node at (270:0.75) {$1111$};
    \draw node at (270:1.75) {$1001$};

    \draw node at (175:2) {$1000$};
    \draw node at (5:2) {$0001$};

    \draw node at (130:1.3) {$0100$};
    \draw node at (50:1.3) {$0010$};

    \draw node at (160:1.3) {$1100$};
    \draw node at (20:1.3) {$0011$};

    \draw node at (190:0.8) {$1110$};
    \draw node at (350:0.8) {$0111$};

    \draw node at (225:1.5) {$1010$};
    \draw node at (305:1.5) {$0101$};

    \draw node at (255:1.3) {$1011$};
    \draw node at (285:1.3) {$1101$};

    \draw (-3,-2.4) rectangle (3,2);
    
    \end{tikzpicture}

    \caption{Labels of Venn regions by bitstrings for Venn diagrams of two, three, and four sets.}
    \label{fig:venn-labels}
\end{figure}
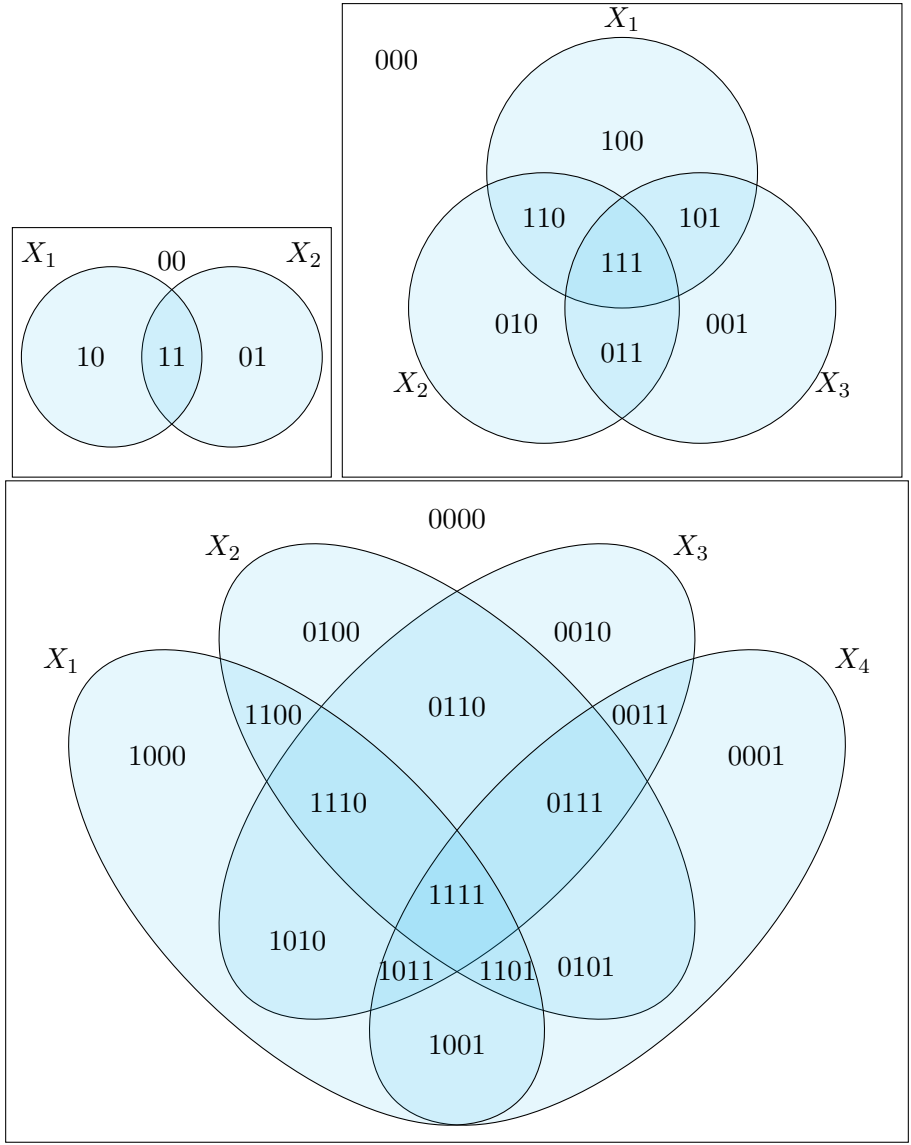

\begin{remark}
    For the remainder of this paper, we will focus on the case where $\mathcal{V} \subseteq \pow{S}$ is a linear subspace. We will denote this condition by $\mathcal{V} \leq \pow{S}$.
\end{remark}

\begin{lemma} \label{lem:basis-is-enough}
    Let $S$ be a finite set, let $\mathcal{V} \leq \pow{S}$, and let $\mathcal{U} \subseteq \mathcal{V}$ be a spanning set for $\mathcal{V}$. Then
    \[
        S / {\sim_{\mathcal{U}}} = S / {\sim_{\mathcal{V}}}.
    \]
\end{lemma}

\begin{proof}
    We will show that $\sim_{\mathcal{U}}$ and $\sim_{\mathcal{V}}$ are the same relation.
    Let $x,y \in S$.

    Suppose $x \sim_{\mathcal{V}} y$. Then $x$ and $y$ are contained in precisely the same elements of $\mathcal{V}$. Since $\mathcal{U} \subseteq \mathcal{V}$, it follows that $x \sim_{\mathcal{U}} y$.

    Suppose $x \sim_{\mathcal{U}} y$. Let $A \in \mathcal{V}$. Since $\mathcal{U}$ spans $\mathcal{V}$, there are sets $X_1, \ldots, X_r \in \mathcal{U}$ such that
    \[
        A = X_1 \symdif \cdots \symdif X_r.
    \]
    Since $x \sim_{\mathcal{U}} y$, we know that $x \in X_i$ if and only if $y \in X_i$ for $1 \leq i \leq r$. This implies that $x \in A$ if and only if $y \in A$. Since this statement holds for all $A \in \mathcal{V}$, it follows that $x \sim_{\mathcal{V}} y$.
\end{proof}

\begin{example} \label{running-example-3}
    Consider the Venn regions of the set
    \[
        S = \{\vv{a}_1, \ldots, \vv{a}_{10}, \vv{w}_1, \vv{w}_2, \vv{w}_3 \} \subseteq \F_2^n
    \]
    from Example~\ref{running-example-1}, with $\mathscr{X} = \{X_1, X_2, X_3\} \subseteq \pow{S}$ defined by
    \begin{align*}
        X_1 &= \{ \vv{w}_1, \vv{a}_1, \ldots, \vv{a}_7 \}, \\
        X_2 &= \{ \vv{w}_2, \vv{a}_4, \ldots, \vv{a}_{10} \}, \\
        X_3 &= \{ \vv{w}_3, \vv{a}_3, \vv{a}_4, \vv{a}_8 \}.
    \end{align*}
    Because $X_1 \symdif X_2 \neq X_3$, the set $\mathscr{X}$ is linear independent. Let $\mathcal{V} = \spann(\mathscr{X})$, so that $\mathscr{X}$ is a basis for $\mathcal{V}$.
    Then $\mathcal{V}$ consists of $\varnothing$, $X_1$, $X_2$, and $X_3$, along with
    \begin{align*}
        X_1 \symdif X_2 &= \{ \vv{w}_1, \vv{w}_2, \vv{a}_1, \vv{a}_2, \vv{a}_3, \vv{a}_8, \vv{a}_9, \vv{a}_{10} \}, \\
        X_1 \symdif X_3 &= \{ \vv{w}_1, \vv{w}_3, \vv{a}_1, \vv{a}_2, \vv{a}_5, \vv{a}_6, \vv{a}_7, \vv{a}_8 \}, \\
        X_2 \symdif X_3 &= \{ \vv{w}_2, \vv{w}_3, \vv{a}_3, \vv{a}_5, \vv{a}_6, \vv{a}_7, \vv{a}_9, \vv{a}_{10} \}, \\
        X_1 \symdif X_2 \symdif X_3 &= \{ \vv{w}_1, \vv{w}_2, \vv{w}_3, \vv{a}_1, \vv{a}_2, \vv{a}_4, \vv{a}_9, \vv{a}_{10} \}.
    \end{align*}
    In Section~\ref{sec:zero-sums} we will show that $\mathcal{V} = E(S)$.
    
    Now let $\mathscr{Y} = \{ Y_1, Y_2, Y_3 \}$, where $Y_1 = X_1 \symdif X_2$, $Y_2 = X_2 \symdif X_3$, and $Y_3 = X_1 \symdif X_2 \symdif X_3$.
    It is straightforward to check that $\mathscr{Y}$ is also a basis for $\mathcal{V}$.

    The Venn diagrams of $S$ with respect to $\mathscr{X}$ and $\mathscr{Y}$ are displayed in Figure~\ref{fig:running-venns}. To find the corresponding Venn regions, one can either compute them directly or capture them from the Venn diagrams.
    \[
    \begin{array}{rcl}
        X_1^{\comp} \cap X_2^{\comp} \cap X_3^{\comp}
        = & \varnothing & =
        Y_1^{\comp} \cap Y_2^{\comp} \cap Y_3^{\comp},
        \\[1ex]
        X_1 \cap X_2^{\comp} \cap X_3^{\comp}
        = & \{ \vv{w}_1, \vv{a}_1, \vv{a}_2 \} & =
        Y_1 \cap Y_2^{\comp} \cap Y_3,
        \\[1ex]
        X_1^{\comp} \cap X_2 \cap X_3^{\comp}
        = & \{ \vv{w}_2, \vv{a}_9, \vv{a}_{10} \} & =
        Y_1 \cap Y_2 \cap Y_3,
        \\[1ex]
        X_1^{\comp} \cap X_2^{\comp} \cap X_3
        = & \{ \vv{w}_3 \} & =
        Y_1^{\comp} \cap Y_2 \cap Y_3,
        \\[1ex]
        X_1 \cap X_2 \cap X_3^{\comp}
        = & \{ \vv{a}_5, \vv{a}_6, \vv{a}_7 \} & =
        Y_1^{\comp} \cap Y_2 \cap Y_3^{\comp},
        \\[1ex]
        X_1 \cap X_2^{\comp} \cap X_3
        = & \{ \vv{a}_3 \} & =
        Y_1 \cap Y_2 \cap Y_3^{\comp},
        \\[1ex]
        X_1^{\comp} \cap X_2 \cap X_3
        = & \{ \vv{a}_8 \} & =
        Y_1 \cap Y_2^{\comp} \cap Y_3^{\comp},
        \\[1ex]
        X_1 \cap X_2 \cap X_3
        = & \{ \vv{a}_4 \} & =
        Y_1^{\comp} \cap Y_2^{\comp} \cap Y_3.
    \end{array}
    \]

    Observe that the elements of $S$ are in different positions in the two Venn diagrams, but the grouping of elements of $S$ into Venn regions is the same for both bases.
    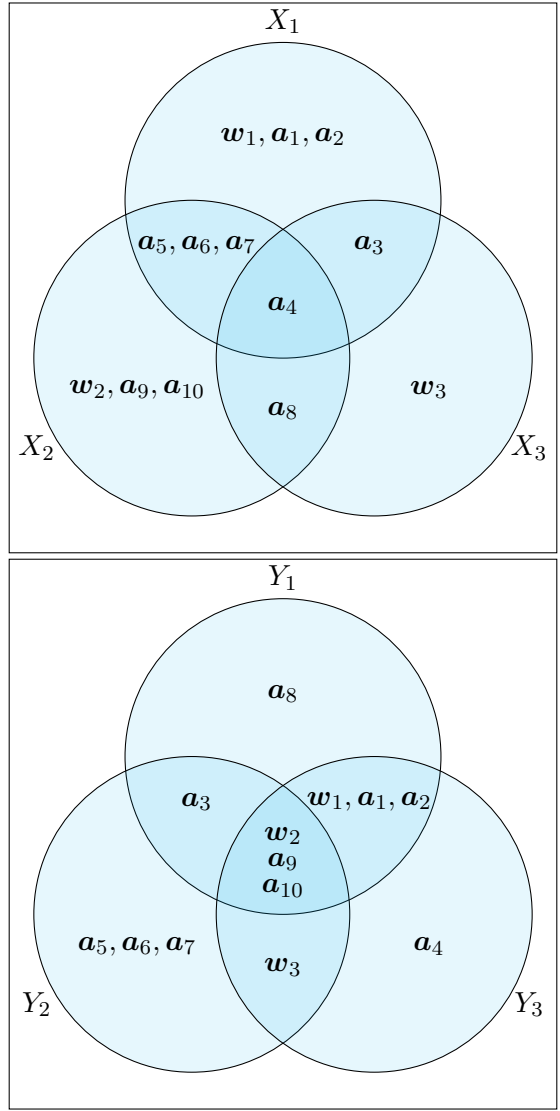
\begin{figure}
        \centering
        \begin{tikzpicture}[set/.style={fill=cyan,fill opacity=0.1}, scale = 0.7]
            \def\firstcircle{(90:2cm) circle (3cm)}
            \def\secondcircle{(210:2cm) circle (3cm)}
            \def\thirdcircle{(330:2cm) circle (3cm)}

            \draw (-5.2cm,-4.7cm) rectangle (5.2cm,5.75cm);
            \draw[set] \firstcircle;
            \draw[set] \secondcircle;
            \draw[set] \thirdcircle;

            \draw node at (90:5.4cm) {$X_1$};    
            \draw node at (210:5.4cm) {$X_2$};    
            \draw node at (330:5.4cm) {$X_3$};    
        
                \draw node at (90:3.2cm) {$\vv{w}_1, \vv{a}_1, \vv{a}_2$};    
                \draw node at (210:3.2cm) {$\vv{w}_2, \vv{a}_9, \vv{a}_{10}$};
                \draw node at (330:3.2cm) {$\vv{w}_3$};
                \draw node at (0,0) {$\vv{a}_4$};
                \draw node at (270:2cm) {$\vv{a}_8$};
                \draw node at (35:2cm) {$\vv{a}_3$};
                \draw node at (145:2cm) {$\vv{a}_5, \vv{a}_6, \vv{a}_7$};
        \end{tikzpicture}
        \begin{tikzpicture}[set/.style={fill=cyan,fill opacity=0.1}, scale = 0.7]
            \def\firstcircle{(90:2cm) circle (3cm)}
            \def\secondcircle{(210:2cm) circle (3cm)}
            \def\thirdcircle{(330:2cm) circle (3cm)}

            \draw (-5.2cm,-4.7cm) rectangle (5.2cm,5.75cm);
            \draw[set] \firstcircle;
            \draw[set] \secondcircle;
            \draw[set] \thirdcircle;

            \draw node at (90:5.4cm) {$Y_1$};    
            \draw node at (210:5.4cm) {$Y_2$};    
            \draw node at (330:5.4cm) {$Y_3$};    
        
                \draw node at (90:3.2cm) {$\vv{a}_8$};    
                \draw node at (210:3.2cm) {$\vv{a}_5, \vv{a}_6, \vv{a}_7$};
                \draw node at (330:3.2cm) {$\vv{a}_4$};
                \draw node at (90:0.5cm) {$\vv{w}_2$};
                \draw node at (0,0) {$\vv{a}_9$};
                \draw node at (270:0.5cm) {$\vv{a}_{10}$};
                \draw node at (270:2cm) {$\vv{w}_3$};
                \draw node at (35:2cm) {$\vv{w}_1, \vv{a}_1, \vv{a}_2$};
                \draw node at (145:2cm) {$\vv{a}_3$};
        \end{tikzpicture}
    
        \caption{The Venn diagrams corresponding to the bases $\mathscr{X}$ and $\mathscr{Y}$ for $E(S)$ in Example~\ref{running-example-3}.}
        \label{fig:running-venns}
    \end{figure}

\end{example}

\begin{definition}
    Let $S$ be a set, let $\mathcal{V} \leq \pow{S}$, and let $\mathscr{X} = \{X_1, \ldots, X_r\}$ be a basis for $\mathcal{V}$.
    The \deff{coordinate function for $\mathcal{V}$ with respect to $\mathscr{X}$} is $\func{\zmap_{\mathscr{X}}}{\F_2^r}{\mathcal{V}}$ defined by
    \[
        \zmap_{\mathscr{X}}(\vv{a})
        = \bigsymdif_{i \in I_{\vv{a}}} X_i
    \]
    for $\vv{a} \in Z_2^r$, where $I_{\vv{a}} \subseteq \{1, \ldots, r\}$ is the support of $\vv{a}$.
    
    The set of \deff{Venn regions of $S$ with respect to $\mathscr{X}$} is denoted $\venn_{\mathscr{X}}(S, \mathcal{V})$. The \deff{coordinate function for $\venn_{\mathscr{X}}(S, \mathcal{V})$} is $\func{\vmap_{\mathscr{X}}}{\F_2^r}{\venn_{\mathscr{X}}(S, \mathcal{V})}$ defined by
    \[
        \vmap_{\mathscr{X}}(\vv{a})
        = \left( \bigcap_{i \in I_{\vv{a}}} X_i \right) \cap \left( \bigcap_{i \notin I_{\vv{a}}} X_i^{\comp} \right)
    \]
    for $\vv{a} \in \F_2^r$.
\end{definition}

\begin{remark} \leavevmode
\begin{enumerate}
    \item The function $\func{\zmap_{\mathscr{X}}}{\F_2^r}{\mathcal{V}}$ defined above is the inverse of the map assigning each element of $\mathcal{V}$ to its vector expression with respect to the basis $\mathscr{X}$, so it is a linear bijection.
    
    \item The function $\func{\vmap_{\mathscr{X}}}{\F_2^r}{\venn_{\mathscr{X}}(S, \mathcal{V})}$ is bijective and makes $\venn_{\mathscr{X}}(S, \mathcal{V})$ into a vector space over $\F_2$ of dimension $r$ with addition defined by
    \[
        \vmap_{\mathscr{X}}(\vv{a}) + \vmap_{\mathscr{X}}(\vv{b})
        := \vmap_{\mathscr{X}}(\vv{a} + \vv{b})
    \]
    for all $\vmap_{\mathscr{X}}(\vv{a}), \vmap_{\mathscr{X}}(\vv{b}) \in \venn_{\mathscr{X}}(S, \mathcal{V})$.

    \item The Venn region $\vmap_{\mathscr{X}}(\vv{0})$ is the set of isolated points for $\mathcal{V}$, so $\vmap_{\mathscr{X}}(\vv{0}) = R_{\varnothing}$.

    \item The set $\venn_{\mathscr{X}}(S, \mathcal{V})$, considered as a collection of sets, may contain multiple copies of the empty set, but since these Venn regions are indexed by vectors in $\F_2^r$, they are distinguishable.

    \item Lemma~\ref{lem:basis-is-enough} implies that the elements of $\venn_{\mathscr{X}}(S, \mathcal{V})$ and those of $\venn_{\mathscr{Y}}(S, \mathcal{V})$ are the same for any bases $\mathscr{X}$ and $\mathscr{Y}$ for $\mathcal{V}$, so we will denote this basis-independent collection of Venn regions by $\venn_{\mathcal{V}}(S)$. In Corollary~\ref{cor:independent-addition}, we will show that the vector space structure on $\venn_{\mathcal{V}}(S)$ is also basis-independent.
\end{enumerate}
\end{remark}

By Lemma~\ref{lem:venn-containment}, each set in $\mathcal{V}$ is a union of Venn regions. In the next example, we examine the coordinates for a set and its constituent Venn regions.

\begin{example} \label{running-example-4}
    Consider again the set $S = \{\vv{a}_1, \ldots, \vv{a}_{10}, \vv{w}_1, \vv{w}_2, \vv{w}_3 \}$ from Example~\ref{running-example-1}, with the subset $\mathcal{V} \subseteq \pow{S}$ spanned by the basis $\mathscr{X} = \{X_1, X_2, X_3\}$ defined in Example~\ref{running-example-3}.
    The top of Figure~\ref{fig:running-venns} displays the Venn diagram corresponding to $\mathscr{X}$.
    Table~\ref{running-bitstring-table} displays the elements of $\mathcal{V}$ and the elements of $\venn_{\mathcal{V}}(S)$ parametrized by elements of $\F_2^3$.

    \begin{table}
        \centering
        \caption{The elements of $\mathcal{V}$ and the elements of $\venn_{\mathcal{V}}(S)$ parametrized by elements of $\F_2^3$, from Example~\ref{running-example-4}.}
        \label{running-bitstring-table}
        {\tiny
        \[
        \begin{array}{|l|l|l|}
            \hline
            \vv{c} & \zmap_{\mathscr{X}}(\vv{c}) & \vmap_{\mathscr{X}}(\vv{c}) \\
            \hline \hline
            000 & \varnothing & X_1^{\comp} \cap X_2^{\comp} \cap X_3^{\comp} = \varnothing \\[0.5ex]
            \hline 
            100 & X_1 = \{ \vv{w}_1, \vv{a}_1, \ldots, \vv{a}_7 \} & X_1 \cap X_2^{\comp} \cap X_3^{\comp} = \{ \vv{w}_1, \vv{a}_1, \vv{a}_2 \} \\[1ex]
            \hline 
            010 & X_2 = \{ \vv{w}_2, \vv{a}_4, \ldots, \vv{a}_{10} \} & X_1^{\comp} \cap X_2 \cap X_3^{\comp} = \{ \vv{w}_2, \vv{a}_9, \vv{a}_{10} \} \\[1ex]
            \hline 
            001 & X_3 = \{ \vv{w}_3, \vv{a}_3, \vv{a}_4, \vv{a}_8 \} & X_1^{\comp} \cap X_2^{\comp} \cap X_3 = \{ \vv{w}_3 \} \\[1ex]
            \hline 
            110 & X_1 \symdif X_2 = \{ \vv{w}_1, \vv{w}_2, \vv{a}_1, \vv{a}_2, \vv{a}_3, \vv{a}_8, \vv{a}_9, \vv{a}_{10} \} & X_1 \cap X_2 \cap X_3^{\comp} = \{ \vv{a}_5, \vv{a}_6, \vv{a}_7 \} \\[1ex]
            \hline 
            101 & X_1 \symdif X_3 = \{ \vv{w}_1, \vv{w}_3, \vv{a}_1, \vv{a}_2, \vv{a}_5, \vv{a}_6, \vv{a}_7, \vv{a}_8 \} & X_1 \cap X_2^{\comp} \cap X_3 = \{ \vv{a}_3 \} \\[1ex]
            \hline 
            011 & X_2 \symdif X_3 = \{ \vv{w}_2, \vv{w}_3, \vv{a}_3, \vv{a}_5, \vv{a}_6, \vv{a}_7, \vv{a}_9, \vv{a}_{10} \} & X_1^{\comp} \cap X_2 \cap X_3 = \{ \vv{a}_8 \} \\[1ex]
            \hline 
            111 & X_1 \symdif X_2 \symdif X_3 = \{ \vv{w}_1, \vv{w}_2, \vv{w}_3, \vv{a}_1, \vv{a}_2, \vv{a}_4, \vv{a}_9, \vv{a}_{10} \} & X_1 \cap X_2 \cap X_3 = \{ \vv{a}_4 \} \\[0.5ex]
            \hline
        \end{array}
        \]
        }
    \end{table}

    Note that 
    \begin{align*}
        \zmap_{\mathscr{X}}(101)
        &= X_1 \symdif X_3 \\
        &= \{ \vv{w}_1, \vv{w}_3, \vv{a}_1, \vv{a}_2, \vv{a}_5, \vv{a}_6, \vv{a}_7, \vv{a}_8 \} \\
        &= \{ \vv{w}_1, \vv{a}_1, \vv{a}_2 \}
            \cup \{ \vv{a}_5, \vv{a}_6, \vv{a}_7 \}
            \cup \{ \vv{w}_3 \}
            \cup \{ \vv{a}_8 \} \\
        &= \vmap_{\mathscr{X}}(100) \cup \vmap_{\mathscr{X}}(110) \cup \vmap_{\mathscr{X}}(001) \cup \vmap_{\mathscr{X}}(011).
    \end{align*}
    Observe that $101$ has dot product $1$ with each of $100, 110, 001, 011$, and in fact these are the only solutions to $101 \cdot \vv{b} = 1$. One can check that this same property holds for $\zmap_{\mathscr{X}}(\vv{a})$ for any $\vv{a} \in \F_2^3$. This is the first half of Proposition~\ref{prop:intersect-and-union} below.
    Further foreshadowing this proposition, note also that 
    \[
        \vmap_{\mathscr{X}}(101)
        = \{ \vv{a}_3 \}
        = \zmap_{\mathscr{X}}(100) \cap \zmap_{\mathscr{X}}(110) \cap \zmap_{\mathscr{X}}(001) \cap \zmap_{\mathscr{X}}(011).
    \]
\end{example}

\begin{lemma} \label{prop:venn-in-zero-sum}
    Let $S$ be a finite set and let $\mathcal{V} \leq \pow{S}$ have basis $\mathscr{X}$. Let $\vv{a}, \vv{b} \in \F_2^r$ be such that $\vmap_{\mathscr{X}}(\vv{a}) \in \venn_{\mathscr{X}}(S, \mathcal{V})$ is nonempty.
    Then
    \[
        \vmap_{\mathscr{X}}(\vv{a}) \subseteq \zmap_{\mathscr{X}}(\vv{b}) \iff \vv{a} \cdot \vv{b} = 1.
    \]
\end{lemma}

\begin{proof}
    Let $\vv{b} = (b_1, \ldots, b_r) \neq \vv{0}$ have support $I = \{i_1, \ldots, i_s\}$ of $\vv{b}$.
    The elements of $\zmap(\vv{b}) = X_{i_1} \symdif \cdots \symdif X_{i_s}$ are those contained in exactly an odd number of the sets $X_{i_1}, \ldots, X_{i_s}$. We will determine which Venn regions consist of points with that property.

    By definition $\vmap(\vv{a}) \subseteq X_{i_j}$ if and only if $a_{i_j} = 1$, for $1 \leq j \leq s$. Therefore $\vmap(\vv{a})$ is contained in an odd number of the sets $X_{i_1}, \ldots, X_{i_s}$ if and only if the entries $a_{i_1}, \ldots, a_{i_s}$ have an odd number of ones. Now consider the dot product $\vv{a} \cdot \vv{b}$. Considering which entries of $\vv{b}$ are zero and which are one, we have
    \begin{align*}
        \vv{a} \cdot \vv{b}
        &= (a_1, \ldots, a_r) \cdot (b_1, \ldots, b_r) \\[1ex]
        &= a_1 b_1 + \cdots + a_r b_r \\[1ex]
        &= a_{i_1} + \cdots + a_{i_s}.
    \end{align*}
    Since these are elements of $\F_2$, this sum equals one if and only if the entries $a_{i_1}, \ldots, a_{i_s}$ have an odd number of ones. Therefore $\vmap(\vv{a}) \subseteq \zmap(\vv{b})$ if and only if $\vv{a} \cdot \vv{b} = 1$.
\end{proof}

We have the following useful corollary.

\begin{corollary} \label{cor:containment-number}
    Let $S$ be a finite set, and let $\mathcal{V} \leq \pow{S}$ with $\dim(\mathcal{V}) = r$. Each nonempty Venn region $R$ is contained in either none or exactly half, $2^{r-1}$, of the sets in $\mathcal{V}$.
    Hence each point in $S$ is contained in none or exactly $2^{r-1}$ of the sets in $\mathcal{V}$.
    Also each nonempty element of $\mathcal{V}$ contains exactly $2^{r-1}$ nonempty Venn regions.
\end{corollary}

\begin{proof} 
    Let $\mathscr{X}$ be a basis for $\mathcal{V}$, let $R \in \venn_{\mathscr{X}}(S, \mathcal{V})$ be nonempty, and let $\vv{a} \in \F_2^r$ be such that $R = \vmap_{\mathscr{X}}(\vv{a})$. If $\vv{a} = \vv{0}$ then $\vmap(\vv{a}) = R_{\varnothing}$, the isolated point set, so $R$ does not intersect any sets in $\mathcal{V}$. Suppose $\vv{a} \neq \vv{0}$.
    Lemma~\ref{prop:venn-in-zero-sum} implies that
    \[
        \abs{\set{A \in \mathcal{V}}{R \subseteq A}}
        = \abs{\set{\vv{b} \in \F_2^r}{\vv{a} \cdot \vv{b} = 1}}.
    \]
    The linear equation $\vv{a} \cdot \vv{x} = 1$ defines a hyperplane of $\F_2^r$, which is an $(r-1)$-flat and hence contains $2^{r-1}$ points.
    Therefore
    \[
        \abs{\set{A \in \mathcal{V}}{R \subseteq A}}
        = \abs{\set{\vv{b} \in \F_2^r}{\vv{a} \cdot \vv{b} = 1}}
        = 2^{r-1}.
    \]

    Let $\vv{x} \in S$ and let $R_{\vv{x}}$ be the Venn region containing $\vv{x}$. By Lemma~\ref{lem:venn-containment}, we know $R_{\vv{x}}$ is contained in a set in $\mathcal{V}$ if and only if it intersects that set. It follows that $\vv{x}$ is contained in the same number of sets in $\mathcal{V}$, which we just proved is $0$ or $2^{r-1}$.

    Now let $X \in \mathcal{V}$ be nonempty, and let $\vv{b} \in \F_2^r$ be such that $\zmap_{\mathscr{X}}(\vv{b}) = X$, and note that $\vv{b} \neq \vv{0}$. By Lemma~\ref{lem:venn-containment}, the number of nonempty Venn regions contained in $X$ is
    \[
        \abs{\set{\vv{a} \in \F_2^r}{\vv{a} \cdot \vv{b} = 1}}
        = 2^{r-1}.
    \]
\end{proof}

\begin{proposition} \label{prop:intersect-and-union}
    Let $S$ be a set, let $\mathcal{V} \leq \pow{S}$, let $A \in \mathcal{V}$, and let $R \in S / {\sim_{\mathcal{V}}} - \{R_{\varnothing}\}$.
    Then
    \[
        A = \bigcup_{\substack{R' \in S/{\sim_{\mathcal{V}}}, \\ R' \cap A \neq \varnothing}} R'
        \quad \text{and} \quad
        R = \bigcap_{\substack{A' \in \mathcal{V}, \\ R \cap A' \neq \varnothing}} A'.
    \]
\end{proposition}

\begin{proof}
    By Lemma~\ref{lem:venn-containment}, we know that $A$ is the union of the Venn regions that intersect it, so
    \[
        A = \bigcup_{\substack{R' \in S/{\sim_{\mathcal{V}}}, \\ R' \cap A \neq \varnothing}} R'.
    \]
    
    Let $\mathcal{V}_R = \set{A' \in \mathcal{V}}{R \cap A' \neq \varnothing}$. Corollary~\ref{cor:containment-number} implies that $\abs{\mathcal{V}_R} = 2^{r-1}$, where $r = \dim(\mathcal{V})$.
    By the definition of Venn regions we have
    \[
        R =
        \bigcap_{A' \in \mathcal{V}_R} A'
        - \bigcap_{A' \in (\mathcal{V}_R)^{\comp}} A'
        \subseteq
        \bigcap_{A' \in \mathcal{V}_R} A'.
    \]
    On the other hand, let $\vv{x} \in \bigcap_{A \in \mathcal{V}_R} A'$. Since $\abs{\mathcal{V}_R} = 2^{r-1}$ and by Corollary~\ref{cor:containment-number} we know $\vv{x}$ is contained in exactly $2^{r-1}$ elements of $\mathcal{V}$, it follows that $\vv{x} \notin A'$ for all $A' \in (\mathcal{V}_R)^{\comp}$. Thus
    \[
        \vv{x} \in
        \bigcap_{A' \in \mathcal{V}_R} A'
        - \bigcap_{A' \in (\mathcal{V}_R)^{\comp}} A'
        = R.
    \]
    Therefore $R = \bigcap_{A' \in \mathcal{V}_R} A'$.
\end{proof}


For any change of basis $\mathscr{X} \mapsto \mathscr{Y}$ of a vector space $\mathcal{V} \subseteq \pow{S}$, there is a corresponding change of coordinates of $\venn_{\mathcal{V}}(S)$. We now prove that this change of coordinates is linear, and easily derived from the change of basis for $\mathcal{V}$.

\begin{theorem} \label{thm:venn-transformation-formula}
Let $S$ be a finite set, let $\mathcal{V} \leq \pow{S}$ have $\dim (\mathcal{V}) = r$, and let $\mathscr{X} = \{X_1, \ldots, X_r\}$ and $\mathscr{Y} = \{Y_1, \ldots, Y_r\}$ be bases for $\mathcal{V}$ with change of basis matrix $M$ from $\mathscr{X}$ to $\mathscr{Y}$. Then the corresponding coordinates for $\venn_{\mathcal{V}}(S)$ transform by the matrix $(M^T)^{-1}$, i.e.~if the commutative diagram below on the left commutes, then so does the diagram on the right.
\[
\begin{tikzcd}
    \F_2^r \arrow[rrd, "\zmap_{\mathscr{X}}"'] \arrow[rr, "\text{mult.~by $M$}"] & & \F_2^r \arrow[d, "\zmap_{\mathscr{Y}}"] \\
    & & \mathcal{V}
\end{tikzcd}
    \qquad \Longrightarrow \qquad
\begin{tikzcd}
    \F_2^r \arrow[rrd, "\vmap_{\mathscr{X}}"'] \arrow[rr, "\text{mult.~by $(M^T)^{-1}$}"] & & \F_2^r \arrow[d, "\vmap_{\mathscr{Y}}"] \\
    & & \venn_{\mathcal{V}}(S)
\end{tikzcd}
\]
\end{theorem}

\begin{proof}
    Let $\{\vv{e}_1, \ldots, \vv{e}_r\}$ be the standard basis for $\F_2^r$.
    Then $\zmap_{\mathscr{X}}(\vv{b}) = \zmap_{\mathscr{Y}} (M \vv{b})$ for all $\vv{b} \in \F_2^r$.
    Since coordinate functions for $\venn_{\mathcal{V}}(S)$ with respect to a basis are bijective linear functions $\F_2^r \to \venn_{\mathcal{V}}(S)$, we have an invertible linear change-of-coordinates function
    \[
        \func{f = (\vmap_{\mathscr{Y}})^{-1} \circ \vmap_{\mathscr{X}}}{\F_2^r}{\F_2^r}.
    \]
    We will show that $f$ is given by multiplication by the matrix $(M^T)^{-1}$.

    Let $\vv{a}, \vv{b} \in \F_2^r$ be nonzero such that $\vmap_{\mathscr{X}}(\vv{a}) \in \venn_{\mathcal{V}}(S)$ is nonempty.
    Then by Lemma~\ref{prop:venn-in-zero-sum}, we have
    \begin{eqnarray*}
        \vv{a} \cdot \vv{b} = 1
        & \iff &
        \vmap_{\mathscr{X}}(\vv{a}) \subseteq \zmap_{\mathscr{X}}(\vv{b}) \\
        & \iff & 
        \vmap_{\mathscr{Y}}(f(\vv{a})) \subseteq \zmap_{\mathscr{Y}}(M \vv{b}) \\
        & \iff & 
        f(\vv{a}) \cdot (M \vv{b}) = 1.
    \end{eqnarray*}
    Since the only possible outputs of this dot product are $0$ and $1$, it follows that
    \[
        \vv{a} \cdot \vv{b} = f(\vv{a}) \cdot (M \vv{b}).
    \]
    Since every basis vector $\vv{e}_i$ is nonzero and has $\vmap_{\mathscr{X}}(\vv{e}_i)$ nonempty, it follows that
    \[
        \vv{e}_i \cdot \vv{e}_j
        = f(\vv{e}_i) \cdot (M \vv{e}_j)
        = (M^T f(\vv{e}_i)) \cdot \vv{e}_j
    \]
    for all $i, j \in \{1, \ldots, r\}$. Therefore $M^T f(\vv{e}_i) = \vv{e}_i$ for all $i \in \{1, \ldots, r\}$. Since $f$ is linear it follows that $M^T (f(\vv{x})) = \vv{x}$, and hence $f(\vv{x}) = (M^T)^{-1} \vv{x}$, for all $\vv{x} \in \F_2^r$.
\end{proof}

\begin{corollary} \label{cor:independent-addition}
    Let $S$ be a finite set and let $\mathcal{V} \leq \pow{S}$ be a linear subspace. Then the addition operation on $\venn_{\mathcal{V}}(S)$ is independent of the choice of basis for $\mathcal{V}$.
\end{corollary}

\begin{proof}
    Let $\dim(\mathcal{V}) = r$, and let $\mathscr{X}$ and $\mathscr{Y}$ be bases for $\mathcal{V}$. By Theorem~\ref{thm:venn-transformation-formula} there is an $(r \times r)$-matrix $N$ over $\F_2$ such that $\vmap_{\mathscr{X}}(\vv{a}) = \vmap_{\mathscr{Y}}(N \vv{a})$ for all $\vv{a} \in \F_2^r$. Denote the addition of Venn regions induced by $\mathscr{X}$ by $+_{\mathscr{X}}$, and that induced by $\mathscr{Y}$ by $+_{\mathscr{Y}}$.

    Let $R_1, R_2 \in \venn_{\mathcal{V}}(S)$. Then there are $\vv{a}, \vv{b} \in \F_2^r$ such that $R_1 = \vmap_{\mathscr{X}}(\vv{a}) = \vmap_{\mathscr{Y}}(N \vv{a})$ and $R_2 = \vmap_{\mathscr{X}}(\vv{b}) = \vmap_{\mathscr{Y}}(N \vv{b})$.
     Then
    \begin{align*}
        R_1 +_{\mathscr{X}} R_2
        &= \vmap_{\mathscr{X}}(\vv{a}) +_{\mathscr{X}} \vmap_{\mathscr{X}}(\vv{b}) \\
        &= \vmap_{\mathscr{X}}(\vv{a} + \vv{b}) \\
        &= \vmap_{\mathscr{Y}}(N(\vv{a} + \vv{b})) \\
        &= \vmap_{\mathscr{Y}}(N\vv{a} + N\vv{b}) \\
        &= \vmap_{\mathscr{Y}}(N\vv{a}) +_{\mathscr{Y}} \vmap_{\mathscr{Y}}(N\vv{b}) \\
        &= R_1 +_{\mathscr{Y}} R_2.
    \end{align*}
\end{proof}

\begin{example} \label{running-example-5}
    Consider again the set $S = \{\vv{a}_1, \ldots, \vv{a}_{10}, \vv{w}_1, \vv{w}_2, \vv{w}_3 \}$ from Example~\ref{running-example-1}, with the subset $\mathcal{V} \subseteq \pow{S}$ spanned by the basis $\mathscr{X} = \{X_1, X_2, X_3\}$. In Example~\ref{running-example-3} we introduced another basis $\mathscr{Y} = \{Y_1, Y_2, Y_3\}$ for $\mathcal{V}$, where $Y_1 = X_1 \symdif X_2$, $Y_2 = X_2 \symdif X_3$, and $Y_3 = X_1 \symdif X_2 \symdif X_3$. Letting $M$ be the change-of-basis matrix from $\mathscr{X}$ to $\mathscr{Y}$, we have
    \[
        M^{-1} =
        \begin{bmatrix} 
            1 & 0 & 1 \\ 
            1 & 1 & 1 \\ 
            0 & 1 & 1
        \end{bmatrix}.
    \]
    The Venn regions of $S$ are
    \begin{align*}
        \vmap_{\mathscr{X}}(000) = \vmap_{\mathscr{Y}}(000) &= \varnothing,
        \\
        \vmap_{\mathscr{X}}(100) = \vmap_{\mathscr{Y}}(101) &= \{ \vv{w}_1, \vv{a}_1, \vv{a}_2 \},
        \\[1ex]
        \vmap_{\mathscr{X}}(010) = \vmap_{\mathscr{Y}}(111) &= \{ \vv{w}_2, \vv{a}_9, \vv{a}_{10} \},
        \\[1ex]
        \vmap_{\mathscr{X}}(001) = \vmap_{\mathscr{Y}}(011) &= \{ \vv{w}_3 \},
        \\[1ex]
        \vmap_{\mathscr{X}}(110) = \vmap_{\mathscr{Y}}(010) &= \{ \vv{a}_5, \vv{a}_6, \vv{a}_7 \},
        \\[1ex]
        \vmap_{\mathscr{X}}(101) = \vmap_{\mathscr{Y}}(110) &= \{ \vv{a}_3 \},
        \\[1ex]
        \vmap_{\mathscr{X}}(011) = \vmap_{\mathscr{Y}}(100) &= \{ \vv{a}_8 \},
        \\[1ex]
        \vmap_{\mathscr{X}}(001) = \vmap_{\mathscr{Y}}(111) &= \{ \vv{a}_4 \}.
    \end{align*}
    One can check that the matrix for the change-of-basis transformation 
    \[
        \func{(\vmap_{\mathscr{Y}})^{-1} \circ \vmap_{\mathscr{X}}}{\F_2^r}{\F_2^r}
    \]
    is
    \[
        (M^T)^{-1} =
        \begin{bmatrix} 
            1 & 1 & 0 \\
            0 & 1 & 1 \\
            1 & 1 & 1
        \end{bmatrix}.
    \]
\end{example}


\section{Maps between Venn regions} \label{sec:venn-maps}

Any map $\func{g}{S}{T}$ between sets $S$ and $T$ induces a map $\func{g_{\ast}}{\pow{S}}{\pow{T}}$ given by the image of each subset. By Lemma~\ref{symdif-preserved}, if $g$ is injective then $g$ preserves the symmetric difference, so $g_{\ast}$ is linear.

\begin{definition}
    Let $S$ and $T$ be sets, and let $\mathcal{V} \leq \pow{S}$ and $\mathcal{W} \leq \pow{T}$. A function $\func{g}{S}{T}$ is a \deff{$(\mathcal{V}, \mathcal{W})$-isomorphism} if $g$ is bijective and $g_{\ast}(\mathcal{V}) = \mathcal{W}$.
\end{definition}

\begin{remark}
    Note that if $\func{g}{S}{T}$ is a $(\mathcal{V}, \mathcal{W})$-isomorphism, then the restriction $g_{\ast}|_{\mathcal{V}}$ is an injective linear map with image $\mathcal{W}$, so induces a linear bijection
    \[
        \func{g_{\mathcal{V} \mathcal{W}}}{\mathcal{V}}{\mathcal{W}}.
    \]
    Furthermore $g_{\mathcal{V} \mathcal{W}}$ is \deff{cardinality-preserving}, in the sense that
    \[
        \abs{g_{\mathcal{V} \mathcal{W}}(X)} = \abs{X}
    \]
    for all $X \in \mathcal{V}$.
\end{remark}



\begin{proposition} \label{prop:induced-Venn-map}
    Let $S$ and $T$ be sets, let $\mathcal{V} \leq \pow{S}$ and $\mathcal{W} \leq \pow{S}$, and let $\func{g}{S}{T}$ be a $(\mathcal{V}, \mathcal{W})$-isomorphism.
    Then $g_{\ast}|_{\venn(S,T)}$ is a cardinality-preserving linear bijection $\venn(S, \mathcal{V}) \to \venn(T, \mathcal{W})$.
\end{proposition}

\begin{proof}
    Since $g$ is bijective, we know that $g_{\ast}$ and hence $g_{\ast}|_{\venn(S, \mathcal{V})}$ is cardinality-preserving. We now prove that the latter map is linear and injective with image $\venn(T, \mathcal{W})$.

    Let $\mathscr{X} = \{X_1, \ldots, X_r\}$ be a basis for $\mathcal{V}$, and let $Y_i = g_{\mathcal{V} \mathcal{W}}(X_i)$ for $1 \leq i \leq r$. Because $\func{g_{\mathcal{V} \mathcal{W}}}{\mathcal{V}}{\mathcal{W}}$ is a linear bijection, the set $\mathscr{Y} = \{Y_1, \ldots, Y_r\}$ is a basis for $\mathcal{W}$. Since $g$ is bijective and hence preserves set intersection and complement, for any $\vv{a} \in \F_2^r$ with support $I \subseteq \{1, \ldots, r\}$ we have
    \begin{align*}
        (g_{\ast} \circ \vmap_{\mathscr{X}})(\vv{a})
        &= g_{\ast} \left[ \left( \bigcap_{i \in I} X_i \right) \cap \left( \bigcap_{i \notin I} X_i^{\comp} \right) \right] \\
        &= \left( \bigcap_{i \in I} g_{\ast}(X_i) \right) \cap \left( \bigcap_{i \notin I} [g_{\ast}(X_i)]^{\comp} \right) \\
        &= \left( \bigcap_{i \in I} Y_i \right) \cap \left( \bigcap_{i \notin I} Y_i^{\comp} \right) \\
        &= \vmap_{\mathscr{Y}}(\vv{a}).
    \end{align*}
    Hence $g_{\ast} \circ \vmap_{\mathscr{X}} = \vmap_{\mathscr{Y}}$, i.e.~the following diagram commutes:
    \[
    \begin{tikzcd}
        \F_2^r \arrow[rd, "\vmap_{\mathscr{Y}}"'] \arrow[r, "\vmap_{\mathscr{X}}"] & \venn(S, \mathcal{V}) \arrow[d, "g_{\ast}"] \\
        & \venn(T, \mathcal{W})
    \end{tikzcd}
    \]
    Therefore $g_{\ast}|_{\venn(S, \mathcal{V})} = \vmap_{\mathscr{Y}} \circ \vmap_{\mathscr{X}}^{-1}$. Since $\func{\vmap_{\mathscr{X}}}{\F_2^r}{\venn(S, \mathcal{V})}$ and $\func{\vmap_{\mathscr{Y}}}{\F_2^r}{\venn(T, \mathcal{W})}$ are linear bijections, it follows that $g_{\ast}|_{\venn(S, \mathcal{V})}$ defines a linear bijection $\venn(S, \mathcal{V}) \to \venn(T, \mathcal{W})$.
\end{proof}

Proposition~\ref{prop:induced-Venn-map} has the immediate following corollary.

\begin{corollary}
    Let $S$ and $T$ be finite sets, and let $\mathcal{V} \leq \pow{S}$ and $\mathcal{W} \leq \pow{T}$. If there exists a $(\mathcal{V}, \mathcal{W})$-isomorphism $S \to T$, then the multisets of Venn region cardinalities for $\venn(S, \mathcal{V})$ and $\venn(T, \mathcal{W})$ must be the same. Hence, if the multisets of Venn region cardinalities are different, there is no $(\mathcal{V}, \mathcal{W})$-isomorphism.
\end{corollary}

Not only does a $(\mathcal{V},\mathcal{W})$-isomorphism between sets induce a cardinality-preserving isomorphism between Venn region sets, but in fact the converse is true.

\begin{theorem} \label{big_theorem}
     Let $S$ and $T$ be finite sets, and let $\mathcal{V} \leq \pow{S}$ and $\mathcal{W} \leq \pow{T}$. There exists a $(\mathcal{V}, \mathcal{W})$-isomorphism $S \to T$ if and only if there is a cardinality-preserving linear bijection $\venn(S, \mathcal{V}) \to \venn(T, \mathcal{W})$.
\end{theorem}

\begin{proof} \leavevmode

    Suppose $\func{g}{S}{T}$ is a $(\mathcal{V}, \mathcal{W})$-isomorphism.
    By Proposition~\ref{prop:induced-Venn-map}, the map $\func{g_{\ast}|_{\venn(S, \mathcal{V})}}{\venn(S, \mathcal{V})}{\venn(T, \mathcal{W})}$ is a cardinality-preserving linear bijection.        
    
    \bigskip
    
    Suppose there is a cardinality-preserving linear bijection
    \[
        \func{h}{\venn(S, \mathcal{V})}{\venn(T, \mathcal{W})}.
    \]
    For each $R \in \venn(S, \mathcal{V})$, since $h(R)$ has the same cardinality as $R$, we can define a bijection $R \to h(R)$. Doing so for each Venn region of $S$, since $S$ and $T$ are each the disjoint unions of their respective Venn regions, we obtain a bijective function $\func{g}{S}{T}$ such that $g_{\ast}(R) = h(R)$ for each $R \in \venn(S, \mathcal{V})$. We will show that $g$ is a $(\mathcal{V}, \mathcal{W})$-isomorphism.
        
    Let $r = \dim \venn(S, \mathcal{V}) = \dim \venn(T, \mathcal{W})$, and recall that this means $r = \dim(\mathcal{V}) = \dim(\mathcal{W})$, as well. Let $\mathscr{X} = \{X_1, \ldots, X_r\}$ and $\mathscr{Y} = \{Y_1, \ldots, Y_r\}$ be bases for $\mathcal{V}$ and $\mathcal{W}$, respectively. Let $N$ be the matrix representation of $\func{h}{\venn(S, \mathcal{V})}{\venn(T, \mathcal{W})}$ with respect to the coordinates $\vmap_\mathscr{X}$ and $\vmap_\mathscr{Y}$, so that the following diagram commutes:
    \[
    \begin{tikzcd}
        \F_2^r \arrow[d, "\vmap_{\mathscr{X}}"'] \arrow[rr, "\text{mult.~by $N$}"] & & \F_2^r \arrow[d, "\vmap_{\mathscr{Y}}"] \\
        \venn(S, \mathcal{V}) \arrow[rr, "h"] & & \venn(T, \mathcal{W})
    \end{tikzcd}
    \]

    Let $M = (N^{-1})^T$, and let $\func{\gamma}{\mathcal{V}}{\mathcal{W}}$ be the linear bijection represented by the matrix $M$ with respect to the bases $\mathscr{X}$ and $\mathscr{Y}$, so that the following diagram commutes:
    \[
    \begin{tikzcd}
        \F_2^r \arrow[d, "\zmap_{\mathscr{X}}"'] \arrow[rr, "\text{mult.~by $M$}"] & & \F_2^r \arrow[d, "\zmap_{\mathscr{Y}}"] \\
            \mathcal{V} \arrow[rr, "\gamma"] & & \mathcal{W}
    \end{tikzcd}
    \]
    We will prove that $\gamma(X) = g_{\ast}(X)$ for all $X \in \mathcal{V}$, so that $g_{\ast}$ defines a linear bijection $\mathcal{V} \to \mathcal{W}$, so $g$ is a $(\mathcal{V}, \mathcal{W})$-isomorphism.

    Let $X \in \mathcal{V}$. Then $X = \zmap_{\mathscr{X}}(\vv{a})$ for some $\vv{a} \in \F_2^r$. Then by Lemma~\ref{prop:venn-in-zero-sum} and Proposition~\ref{prop:intersect-and-union} we know
    \[
        \gamma(X)
        = \gamma(\zmap_{\mathscr{X}}(\vv{a}))
        = \zmap_{\mathscr{Y}}(M \vv{a})
        = \bigcup_{\substack{\vv{b} \in \F_2^r, \\ (M\vv{a}) \cdot \vv{b} = 1}} \vmap_{\mathscr{Y}}(\vv{b}).
    \]
    Observe that
    \begin{align*}
        \set{\vv{b} \in \F_2^r}{(M \vv{a}) \cdot \vv{b} = 1}
        &= \set{\vv{b} \in \F_2^r}{((N^{-1})^T \vv{a}) \cdot \vv{b} = 1} \\
        &= \set{\vv{b} \in \F_2^r}{\vv{a} \cdot (N^{-1} \vv{b}) = 1}.
    \end{align*}
    Since $N$ is invertible, we can apply the change of variables $\vv{b} = N \vv{c}$ to obtain
    \begin{align*}
        \set{\vv{b} \in \F_2^r}{\vv{a} \cdot (N^{-1} \vv{b}) = 1}
        &= \set{N \vv{c}}{\text{$\vv{c} \in \F_2^r$ and $\vv{a} \cdot (N^{-1} N \vv{c}) = 1$}} \\
        &= \set{N \vv{c}}{\text{$\vv{c} \in \F_2^r$ and $\vv{a} \cdot \vv{c} = 1$}}.
    \end{align*}
    Therefore
    \begin{align*}
        \gamma(X)
        = \bigcup_{\substack{\vv{b} \in \F_2^r, \\ (M \vv{a}) \cdot \vv{b} = 1}} \vmap_{\mathscr{Y}}(\vv{b})
        &= \bigcup_{\substack{\vv{c} \in \F_2^r, \\ \vv{a} \cdot \vv{c} = 1}} \vmap_{\mathscr{Y}}(N \vv{c}) \\
        &= \bigcup_{\substack{\vv{c} \in \F_2^r, \\ \vv{a} \cdot \vv{c} = 1}} h(\vmap_{\mathscr{X}}(\vv{c})) \\
        &= \bigcup_{\substack{\vv{c} \in \F_2^r, \\ \vv{a} \cdot \vv{c} = 1}} g_{\ast}[\vmap_{\mathscr{X}}(\vv{c})] \\
        &= g \left[ \bigcup_{\substack{\vv{c} \in \F_2^r, \\ \vv{a} \cdot \vv{c} = 1}} \vmap_{\mathscr{X}}(\vv{c}) \right] \\
        &= g_{\ast} \left[ \zmap_{\mathscr{X}}(\vv{a}) \right] \\
        &= g_{\ast} (X).
    \end{align*}
    Therefore $g_{\ast}$ defines a linear bijection $\mathcal{V} \to \mathcal{W}$. This completes the proof.
\end{proof}

By Theorem~\ref{big_theorem}, the key information we need from a Venn diagram such as those in Figure~\ref{fig:running-venns} is the cardinalities of each Venn region. For that reason, for a given set $S$ and subspace $\mathcal{V} \leq \pow{S}$ with basis $\mathscr{X}$, we will focus on the \deff{Venn cardinality diagram}, which displays the number of elements in each region rather than the elements themselves. See Example~\ref{ex:venn-count-example} in the next section.




        
    

\section[Zero-sum sets]{Zero-sum sets} \label{sec:zero-sums}

Let $S \subseteq \F_2^n$. In this section we will prove that the even zero-sum set $E(S)$ of $S$ is a linear subspace of $\pow{S}$, describe how an affine basis for $S$ induces a linear basis for $E(S)$, and then apply the results of the preceding section to this context.

\begin{proposition}
    \label{prop:even-is-subspace}
    Let $S \subseteq \F_2^n$. Then $E(S) \leq \pow{S}$.
\end{proposition}

\begin{proof}
    By definition $\varnothing \in E(S)$, so $E(S)$ contains the zero vector of $\pow{S}$. It remains to show that $E(S)$ is closed under the symmetric difference.
    
    Let $X, Y \in E(S)$. Let $A = X - Y$, $B = Y - X$, and $C = X \cap Y$. Then $X \symdif Y = A \sqcup B$. Since $\abs{X} = \abs{A} + \abs{C}$ and $\abs{Y} = \abs{B} + \abs{C}$ are even, it follows that $\abs{A}$ and $\abs{B}$ have the same parity, so
    \[
        \abs{X \symdif Y}
        = \abs{A \sqcup B}
        = \abs{A} + \abs{B}
    \]
    must be even.
    Also
    \begin{align*}
        \sum_{\vv{x} \in X \symdif Y} \vv{x}
        &= \sum_{\vv{x} \in A} \vv{x} + \sum_{\vv{x} \in B} \vv{x} \\
        &= \sum_{\vv{x} \in A} \vv{x} + \sum_{\vv{x} \in B} \vv{x} + 2 \sum_{\vv{x} \in C} \vv{x} \\
        &= \left( \sum_{\vv{x} \in A} \vv{x} + \sum_{\vv{x} \in C} \vv{x} \right) + \left( \sum_{\vv{x} \in B} \vv{x} + \sum_{\vv{x} \in C} \vv{x} \right) \\
        &= \sum_{\vv{x} \in X} \vv{x} + \sum_{\vv{x} \in Y} \vv{x} \\
        &= 0 + 0 = 0,
    \end{align*}
    where the last line follows from the fact that $X, Y \in E(S)$.
    Therefore $X \symdif Y \in E(S)$.
\end{proof}

We now describe how to construct a basis for $E(S)$ from a basis for $S$, which will allow us to determine the dimension of $E(S)$.

\begin{definition} \label{def:dependent-to-zero-sums}
    Let $S = B \sqcup D \subseteq \F_2^n$ have affine basis $B = \{\vv{b}_0, \ldots, \vv{b}_d\}$. Let $\vv{w} \in D$, and let $\vv{w} = \vv{b}_{i_1} + \cdots + \vv{b}_{i_j}$ be the unique expression of $\vv{w}$ as an odd sum of basis elements.
    define
    \[
        X_{\vv{w}}
        = \{ \vv{w}, \vv{b}_{i_1}, \ldots, \vv{b}_{i_j} \}
    \in E(S).
    \]
    The \deff{sumset map induced by $B$} is the map $\func{\varphi}{\pow{D}}{E(S)}$ defined by $\varphi(\varnothing) = \varnothing$ and
    \[
        \varphi(D') = \bigsymdif_{\vv{w} \in D'} X_{\vv{w}}
    \]
    for nonempty $D' \in \pow{D}$.
\end{definition}

\begin{lemma} \label{lem:generating}
    Let $S = B \sqcup D \subseteq \F_2^n$ have affine basis $B$, and let $\func{\varphi}{\pow{D}}{E(S)}$ be the sumset map induced by $B$. Let $D', D_1, D_2 \subseteq D$. Then
    \begin{enumerate}
        \item $D \cap \varphi(D') = D'$;
        \item if $D_1 \cap D_2 = \varnothing$ then $\varphi(D_1 \cup D_2) = \varphi(D_1) \symdif \varphi(D_2)$;
        \item $\varphi$ is bijective and linear.
    \end{enumerate}
\end{lemma}

\begin{proof}
    Note that $D \cap X_{\vv{w}} = \vv{w}$ for each $\vv{w} \in D$. Since intersection distributes over symmetric difference, it follows that
    \[
        D \cap \varphi(D')
        = D \cap \left( \bigsymdif_{\vv{w} \in D'} X_{\vv{w}} \right)
        = \bigsymdif_{\vv{w} \in D'} \left( D \cap X_{\vv{w}} \right)
        = \bigsymdif_{\vv{w} \in D'} \{ \vv{w} \}
        = D'.
    \]
        
    Suppose that $D_1 \cap D_2 = \varnothing$. Then
    \begin{align*}
        \varphi(D_1 \cup D_2)
        &= \bigsymdif_{\vv{w} \in D_1 \cup D_2} X_{\vv{w}} \\
        &= \left( \bigsymdif_{\vv{w} \in D_1} X_{\vv{w}} \right) \symdif \left( \bigsymdif_{\vv{w} \in D_2} X_{\vv{w}} \right) \\
        &= \varphi(D_1) \symdif \varphi(D_2).
    \end{align*}

    Suppose that $\varphi(D_1) = \varphi(D_2)$. By part 1 of this lemma, it follows that
    \[
        D_1
        = \varphi(D_1) \cap D
        = \varphi(D_2) \cap D
        = D_2.
    \]
    Hence $\varphi$ is injective.

    Let $Z \in E(S)$ and set $D' = Z \cap D$. Then
    \[
        [Z \symdif \varphi(D')] \cap D
        = (Z \cap D) \symdif [\varphi(D') \cap D]
        = D' \symdif D'
        = \varnothing.
    \]
    Since $S = B \sqcup D$ it follows that $Z \symdif \varphi(D') \subseteq B$. Since $Z \symdif \varphi(D') \in E(S)$ and $E(S) \cap \pow{B} = \varnothing$, this implies that $Z \symdif \varphi(D') = \varnothing$.
    Thus $\varphi(D') = Z$. Therefore $\varphi$ is surjective.

    Finally, it follows from part 2 of this lemma that
    \begin{align*}
        & \varphi(D_1) \symdif \varphi(D_2) \\
        & \qquad = \varphi \left[ (D_1 - D_2) \sqcup (D_1 \cap D_2) \right] \, \symdif \, \varphi \left[ (D_2 - D_1) \sqcup (D_1 \cap D_2) \right] \\
        & \qquad = \varphi(D_1 - D_2) \symdif \varphi(D_1 \cap D_2) \symdif \varphi(D_2 - D_1) \symdif \varphi(D_1 \cap D_2) \\
        & \qquad = \varphi(D_1 - D_2) \symdif \varphi(D_2 - D_1) \\
        & \qquad = \varphi \left( (D_1-D_2) \cup (D_2-D_1) \right) \\
        & \qquad = \varphi(D_1 \symdif D_2).
    \end{align*}
    Hence $\varphi$ is linear.
\end{proof}

\begin{corollary} \label{prop:zero-sum-basis}
    Let $S = B \sqcup D \subseteq \F_2^n$ have affine basis $B$, and let $\abs{S} = k$ and $\dim(S) = d$.
    Then
    \[
        \set{X_{\vv{w}}}{\vv{w} \in D}
    \]
    is a basis for $E(S)$ and $\dim E(S) = k - (d+1)$. 
\end{corollary}

\begin{proof}
    Let $\func{\varphi}{\pow{D}}{E(S)}$ be the sumset map of $E(S)$ with respect to $B$.
    By Lemma~\ref{lem:generating} we know that $\varphi$ is a linear bijection. By Proposition~\ref{prop:bases}, the set
    \[
        \set{\{\vv{w}\}}{\vv{w} \in D}
    \]
    is a basis for $\pow{D}$, and hence
    \[
        \varphi(\set{\{\vv{w}\}}{\vv{w} \in D})
        = \set{X_{\vv{w}}}{\vv{w} \in D}
    \]
    is a basis for $E(S)$. Since $\abs{B} = d+1$, it follows that
    \[
        \dim E(S) = \abs{D} = \abs{S} - \abs{B} = k - (d+1).
    \]
\end{proof}

\begin{example} \label{running-example-2}
    Consider $S = B \sqcup D = \{ \vv{a}_1, \ldots, \vv{a}_{10} \} \sqcup \{ \vv{w}_1, \vv{w}_2, \vv{w}_3 \}$ from Example~\ref{running-example-1}. The subsets
    \begin{align*}
        X_1 &= X_{\vv{w}_1} = \{ \vv{w}_1, \vv{a}_1, \ldots, \vv{a}_7 \}, \\
        X_2 &= X_{\vv{w}_2} = \{ \vv{w}_2, \vv{a}_4, \ldots, \vv{a}_{10} \}, \\
        X_3 &= X_{\vv{w}_3} = \{ \vv{w}_3, \vv{a}_3, \vv{a}_4, \vv{a}_8 \}
    \end{align*}
    form a basis for $E(S)$, by Corollary~\ref{prop:zero-sum-basis}, and
    \[
        \varphi(\varnothing) = \varnothing, \quad
        \varphi(\{\vv{w}_1\}) = X_1, \quad
        \varphi(\{\vv{w}_2\}) = X_2, \quad
        \varphi(\{\vv{w}_3\}) = X_3
    \]
    and
    \begin{align*}
        \varphi(\{\vv{w}_1, \vv{w}_2\}) &= X_1 \symdif X_2 = \{ \vv{w}_1, \vv{w}_2, \vv{a}_1, \vv{a}_2, \vv{a}_3, \vv{a}_8, \vv{a}_9, \vv{a}_{10} \}, \\
        \varphi(\{\vv{w}_1, \vv{w}_3\}) &= X_1 \symdif X_3 = \{ \vv{w}_1, \vv{w}_3, \vv{a}_1, \vv{a}_2, \vv{a}_5, \vv{a}_6, \vv{a}_7, \vv{a}_8 \}, \\
        \varphi(\{\vv{w}_2, \vv{w}_3\}) &= X_2 \symdif X_3 = \{ \vv{w}_2, \vv{w}_3, \vv{a}_3, \vv{a}_5, \vv{a}_6, \vv{a}_7, \vv{a}_9, \vv{a}_{10} \}, \\
        \varphi(\{\vv{w}_1, \vv{w}_2, \vv{w}_3\}) &= X_1 \symdif X_2 \symdif X_3 = \{ \vv{w}_1, \vv{w}_2, \vv{w}_3, \vv{a}_1, \vv{a}_2, \vv{a}_4, \vv{a}_9, \vv{a}_{10} \}.
    \end{align*}
\end{example}

Suppose $S, T \subseteq \F_2^n$ are affinely equivalent. By Theorem~\ref{thm:preserve-zero-sums}, we know there exists a bijection $\func{g}{S}{T}$ such that $g_{\ast}$ defines a bijection $\func{g_E}{E(S)}{E(T)}$. Because $g$ is bijective, it follows from Lemma~\ref{symdif-preserved} that $g_{\ast}$ and $g_E$ preserve the symmetric difference, and hence both are linear. Therefore $g_E$ is an $(E(S),E(T))$-isomorphism between $S$ and $T$. Therefore we can restate Theorem~\ref{thm:preserve-zero-sums} as follows.

\begin{theorem} \label{key_thm}
    Let $S, T \subseteq \F_2^n$. Then $S$ and $T$ are affinely equivalent if and only if there exists an $(E(S),E(T)$-isomorphism $\func{g}{S}{T}$.
\end{theorem}

Combined with Theorem~\ref{big_theorem}, we have the following characterization of affine equivalence in $\F_2^n$.

\begin{theorem} \label{real_big_theorem}
    Let $S, T \subseteq \F_2^n$. Then $S$ and $T$ are affinely equivalent if and only if there is a cardinality preserving linear bijection $\venn(S, E(S)) \to \venn(T, E(T))$.
\end{theorem}

\begin{example} \label{ex:venn-count-example}
    Consider the set $S$ from our previous examples (labeled $S_1$ below), alongside two other similar subsets of $\F_2^n$. We will use Theorem~\ref{real_big_theorem} to determine which of these sets are affinely equivalent.
    
    Let $B = \{\vv{a}_1, \ldots, \vv{a}_{10}\}$ be an affinely independent set, and let
    \[
        D_1 = \{\vv{w}_1, \vv{w}_2, \vv{w}_3\}, \quad
        D_2 = \{\vv{x}_1, \vv{x}_2, \vv{x}_3\}, \quad
        D_3 = \{\vv{y}_1, \vv{y}_2, \vv{y}_3\},
    \]
    where
    \begin{align*}
        \vv{w}_1 &= \vv{a}_1 + \cdots + \vv{a}_7, \\
        \vv{w}_2 &= \vv{a}_4 + \cdots + \vv{a}_{10}, \\
        \vv{w}_3 &= \vv{a}_3 + \vv{a}_4 + \vv{a}_8, \\[2ex]
        \vv{x}_1 &= \vv{a}_1 + \vv{a}_2 + \vv{a}_3 + \vv{a}_7 + \vv{a}_8 + \vv{a}_9 + \vv{a}_{10}, \\
        \vv{x}_2 &= \vv{a}_3 + \vv{a}_4 + \vv{a}_5 + \vv{a}_6 + \vv{a}_7 + \vv{a}_9 + \vv{a}_{10}, \\
        \vv{x}_3 &= \vv{a}_1 + \vv{a}_2 + \vv{a}_4 + \vv{a}_7 + \vv{a}_8 + \vv{a}_9 + \vv{a}_{10}, \\[2ex]
        \vv{y}_1 &= \vv{a}_1 + \cdots \vv{a}_5, \\
        \vv{y}_2 &= \vv{a}_4 + \vv{a}_5 + \vv{a}_8 + \vv{a}_9 + \vv{a}_{10}, \\
        \vv{y}_3 &= \vv{a}_3 + \vv{a}_4 + \vv{a}_6 + \vv{a}_7 + \vv{a}_8.
    \end{align*}
    Put $S_i = B \sqcup D_i$ for $i = 1,2,3$. In Figure~\ref{fig:venn-count-example} we display the Venn cardinality diagrams for $E(S_1)$, $E(S_2)$, and $E(S_3)$ corresponding to the bases induced by $D_1$, $D_2$, and $D_3$, respectively.
    Observe that the multiset of Venn region cardinalities is the same for all three sets, $\{ 0, 1, 1, 1, 1, 3, 3, 3 \}$.

    \begin{figure}
        \centering
        \begin{tabular}{|c|}
        \hline
        $S_1$ \\
        \hline
        \begin{tikzpicture}[set/.style={fill=cyan,fill opacity=0.1}, scale = 0.4]
            \def\firstcircle{(90:2cm) circle (3cm)}
            \def\secondcircle{(210:2cm) circle (3cm)}
            \def\thirdcircle{(330:2cm) circle (3cm)}

            \draw[set] \firstcircle;
            \draw[set] \secondcircle;
            \draw[set] \thirdcircle;

            \draw node at (90:5.6cm) {$X_1$};    
            \draw node at (210:5.6cm) {$X_2$};    
            \draw node at (330:5.6cm) {$X_3$};    
        
                \draw node at (-5.5cm,4.5cm) {$0$};    
                \draw node at (90:2.7cm) {$3$};    
                \draw node at (210:2.7cm) {$3$};
                \draw node at (330:2.7cm) {$1$};
                \draw node at (0,0) {$1$};
                \draw node at (270:2cm) {$1$};
                \draw node at (30:2cm) {$1$};
                \draw node at (150:2cm) {$3$};
        \end{tikzpicture} \\
        \hline
        \end{tabular}
        \begin{tabular}{|c|}
        \hline
        $S_2$ \\
        \hline
        \begin{tikzpicture}[set/.style={fill=cyan,fill opacity=0.1}, scale = 0.4]
            \def\firstcircle{(90:2cm) circle (3cm)}
            \def\secondcircle{(210:2cm) circle (3cm)}
            \def\thirdcircle{(330:2cm) circle (3cm)}

            \draw[set] \firstcircle;
            \draw[set] \secondcircle;
            \draw[set] \thirdcircle;

            \draw node at (90:5.6cm) {$X_1$};    
            \draw node at (210:5.6cm) {$X_2$};    
            \draw node at (330:5.6cm) {$X_3$};    
        
                \draw node at (-5.5cm,4.5cm) {$0$};    
                \draw node at (90:2.7cm) {$1$};    
                \draw node at (210:2.7cm) {$3$};
                \draw node at (330:2.7cm) {$1$};
                \draw node at (0,0) {$3$};
                \draw node at (270:2cm) {$1$};
                \draw node at (30:2cm) {$3$};
                \draw node at (150:2cm) {$1$};
        \end{tikzpicture} \\
            \hline
        \end{tabular}
        \begin{tabular}{|c|}
        \hline
        $S_3$ \\
        \hline
        \begin{tikzpicture}[set/.style={fill=cyan,fill opacity=0.1}, scale = 0.4]
            \def\firstcircle{(90:2cm) circle (3cm)}
            \def\secondcircle{(210:2cm) circle (3cm)}
            \def\thirdcircle{(330:2cm) circle (3cm)}

            \draw[set] \firstcircle;
            \draw[set] \secondcircle;
            \draw[set] \thirdcircle;

            \draw node at (90:5.6cm) {$X_1$};    
            \draw node at (210:5.6cm) {$X_2$};    
            \draw node at (330:5.6cm) {$X_3$};    
        
                \draw node at (-5.5cm,4.5cm) {$0$};    
                \draw node at (90:2.7cm) {$3$};    
                \draw node at (210:2.7cm) {$3$};
                \draw node at (330:2.7cm) {$3$};
                \draw node at (0,0) {$1$};
                \draw node at (270:2cm) {$1$};
                \draw node at (30:2cm) {$1$};
                \draw node at (150:2cm) {$1$};
        \end{tikzpicture} \\
            \hline
        \end{tabular}

        \caption{The Venn cardinality diagrams for $S_1$, $S_2$, and $S_3$ from Example~\ref{ex:venn-count-example}.}
        \label{fig:venn-count-example}
    \end{figure}
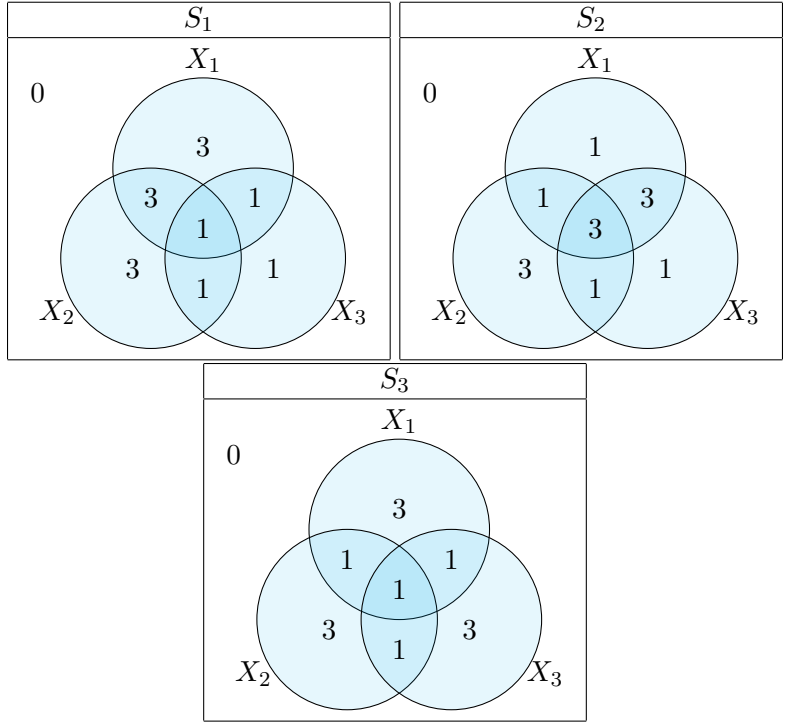

    For $i = 1,2,3$, let $\func{v_i}{\F_2^r}{\venn(S_i, E(S_i))}$ be the coordinate function for $\venn(S_i, E(S_i))$ with respect to the basis for $E(S_i)$ induced by $D_i$. Table~\ref{tab:venn-count-example} lists each Venn cardinality with its corresponding vector.

    \begin{table}
        \centering
        \caption{Venn cardinalities for the sets $S_1$, $S_2$, and $S_3$ in Example~\ref{ex:venn-count-example}.}
        \label{tab:venn-count-example}
        $\begin{array}{|c|ccc|}
            \hline
            & & & \\[-2ex]
            \vv{a} & \abs{v_1(\vv{a})} & \abs{v_2(\vv{a})} & \abs{v_3(\vv{a})} \\[0.5ex]
            \hline
            000 & 0 & 0 & 0 \\
            100 & 3 & 1 & 3 \\
            010 & 3 & 3 & 3 \\
            001 & 1 & 1 & 3 \\
            110 & 3 & 1 & 1 \\
            101 & 1 & 3 & 1 \\
            011 & 1 & 1 & 1 \\
            111 & 1 & 3 & 1 \\
            \hline
       \end{array}$
    \end{table}

    One can check that the matrix
    \[
        M = \begin{bmatrix}
            1 & 0 & 0 \\ 
            0 & 1 & 0 \\ 
            1 & 0 & 1 \end{bmatrix}
    \]
    satisfies $\abs{v_1(\vv{a})} = \abs{v_2(M \vv{a})}$ for all $\vv{a} \in \F_2^r$, so $M$ corresponds to a cardinality-preserving linear bijection $\venn(S_1, E(S_1)) \to \venn(S_2, E(S_2))$. By Theorem~\ref{real_big_theorem} it follows that $S_1$ and $S_2$ are affinely equivalent.
    
    On the other hand, observe that the vectors corresponding to Venn regions of cardinality $3$ for $S_1$ and $S_3$ are $\{100, 010, 110\}$ and $\{100, 010, 001\}$, respectively.
    The first set is dependent but the second is independent. A cardinality-preserving linear bijection $\venn(S_1, E(S_1)) \to \venn(S_3, E(S_3))$ must map the first set to the second set. Since linear bijections preserve dependence and independence, there is no such map $\venn(S_1, E(S_1)) \to \venn(S_3, E(S_3))$. By Theorem~\ref{real_big_theorem}, we conclude that $S_1$ and $S_3$ are not affinely equivalent.
    
\end{example}

\section[Caps with Size-Dimension Difference 3]{Caps with Size-Dimension Difference 3} \label{sec:diff3}

\begin{theorem} \label{thm:diff-3}
    Let $S, T \subseteq \F_2^n$ be sets with size $k$ and dimension $d = k-3$. Then $S$ and $T$ are affinely equivalent if and only if $S$ and $T$ have the same multisets of Venn region cardinalities and the same number of isolated points.
\end{theorem}

\begin{proof}
    If $S$ and $T$ are affinely equivalent, then Theorem~\ref{real_big_theorem} implies that the multisets of cardinalities of elements of $\venn(S, E(S))$ and $\venn(T, E(T))$ are identical, and that $\abs{V_0(S)} = \abs{V_0(T)}$, because these correspond to the zero vector in $\venn(S, \mathcal{V})$ and $\venn(T, \mathcal{W})$, respectively.

    Now, suppose that $S$ and $T$ have the same multisets of Venn region cardinalities and the same number of isolated points. Then there exists a cardinality-preserving bijection $\func{f}{\venn(S, E(S))}{\venn(T, E(T))}$ that maps the isolated point set of $S$ to that of $T$. Both $\venn(S, E(S))$ and $\venn(T, E(T))$ have dimension $k - (d+1) = 2$, so are isomorphic to $\F_2^2$. Since every permutation of $\F_2^2$ that fixes the origin is linear, it follows that the function $f$ must be linear, so $f$ is a cardinality-preserving linear bijection. By Theorem~\ref{real_big_theorem}, this implies that $S$ and $T$ are affinely equivalent.
\end{proof}

\begin{lemma} \label{diff2-cap-lem}
    Let $S \subseteq \F_2^n$ have $\abs{S} = k$ and $\dim(S) = d = k-3$, and let $a \leq b \leq c$ be the Venn region cardinalities of $S$, not including $V_0(S)$. Then $S$ is a cap if and only if the numbers $a+b$, $b+c$, and $a+c$ are all greater than or equal to $6$.
\end{lemma}

\begin{proof}
    Observe that the nonempty even zero-sum sets of $S$ have sizes $a+b$, $b+c$, and $a+c$. These numbers must be even, and none can equal $0$ or $2$, because these are nonempty sets consisting of distinct points. Recall that $S$ is a cap if and only if no four elements of $S$ sum to $0$, which is equivalent to there being no even zero-sum sets of size $4$. Therefore $S$ is a cap if and only if the even numbers $a+b$, $b+c$, and $a+c$ are greater than $4$, which means they are greater than or equal to $6$.
\end{proof}

\begin{theorem} \label{thm:size-dim-3}
    Let $k$ be a positive integer. Nonnegative integers $a \leq b \leq c$ are the Venn region cardinalities (not including $V_0(S)$) for a $k$-cap in $\F_2^n$ with size-dimension difference $3$ if and only if
    \begin{enumerate}
        \item $a + b + c \leq k$;
        \item $a$, $b$, and $c$ all have the same parity;
        \item $a+b \geq 6$.
    \end{enumerate}
\end{theorem}

\begin{proof}
    Let $k$ be a positive integer and $a \leq b \leq c$ be nonnegative integers.
    
    Suppose $a$, $b$, and $c$ are the Venn region cardinalities for a $k$-cap in $\F_2^n$, not including $V_0(S)$. 
    Because the Venn regions partition the cap, we know $a + b + c \leq k$. Since any pair of these Venn regions constitute and even zero-sum, the sums $a+b$, $b+c$, and $a+c$ must all be even, so $a$, $b$, and $c$ must have the same cardinality. Finally, it follows from Lemma~\ref{diff2-cap-lem} that each of $a+b \leq 6$.

    \bigskip
    
    Suppose that:
    \begin{enumerate}
        \item $a + b + c \leq k$;
        \item $a$, $b$, and $c$ all have the same parity;
        \item $a+b \geq 6$.
    \end{enumerate}

    Note that since $a \leq b \leq c$, we have $a+c \geq a+b \geq 6$ and $b+c \geq a+b \geq 6$. Also, if $a = 0$, then $c \geq b \geq 6 - a = 6$, so $b$ and $c$ are both positive.
    Let
    \[
        \vv{x}_1, \ldots, \vv{x}_{b-1}, 
        \vv{y}_1, \ldots, \vv{y}_{c-1}, 
        \vv{z}_1, \ldots, \vv{z}_{a}, 
        \vv{w}_1, \ldots, \vv{w}_{k-(a+b+c)}
        \in \F_2^n
    \]
    be independent, and set
    \begin{align*}
        x_b &= x_1 + \cdots + x_{b-1} + z_1 + \cdots + z_a, \\
        y_c &= y_1 + \cdots + y_{c-1} + z_1 + \cdots + z_a.
    \end{align*}
    Let $V_b = \{x_1, \ldots, x_b\}$, and $V_c = \{y_1, \ldots, y_c\}$, and $V_a = \{z_1, \ldots, z_a\}$, and set $C = V_a \cup V_b \cup V_c \cup \{\vv{w}_1, \ldots, \vv{w}_{k-(a+b+c)}\}$.
    Then $V_a \cup V_b$, $V_a \cup V_c$, and $V_b \cup V_c$ are the nonempty even zero-sum sets of $C$. Since these have cardinality $a+b$, $a+c$, and $b+c$, respectively, and these numbers are all $\geq 6$, we conclude that the set $C$ is a cap.   
    
\end{proof}

Although we do not currently have a closed formula counting the number of equivalence classes of caps of size $k$ and dimension $k-3$, Theorem~\ref{thm:size-dim-3} does make it very easy to write a computer program counting the classes for any particular value of $k$. In Table~\ref{tab:size-dim-3}, we list the number of equivalence classes of $k$-caps with dimension $k-3$ for $8 \leq k \leq 27$. (There are no such caps for $k \leq 8$.)

\begin{table}
    \centering
    \caption{The number of equivalence classes of $k$-caps with dimension $k-3$.}
    \label{tab:size-dim-3}
    \[
    \begin{array}{|c|c|}
        \hline
        k	& \text{\# of classes} \\
        \hline
        8   & 0 \\
        9	& 1 \\
        10	& 2 \\
        11	& 4 \\
        12	& 7 \\
        13	& 10 \\
        14	& 14 \\
        \hline
    \end{array}
    \quad
    \begin{array}{|c|c|}
        \hline
        k	& \text{\# of classes} \\
        \hline
        15	& 19 \\
        16	& 25 \\
        17	& 31 \\
        18	& 39 \\
        19	& 47 \\
        20	& 57 \\
        21  & 67 \\
        \hline
    \end{array}
    \quad
    \begin{array}{|c|c|}
        \hline
        k	& \text{\# of classes} \\
        \hline
        21	& 67 \\
        22	& 79 \\
        23	& 91 \\
        24	& 106 \\
        25	& 120 \\
        26	& 137 \\
        27	& 154 \\
        \hline
    \end{array}
    \]
\end{table}

\section[Caps of dimension 7]{Caps of dimension 7} \label{sec:dim7caps}

In \cite{quad128}, the authors study caps of dimension $7$ in $\F_2^n$ by defining the \deff{extended cap basis type} (\cite[Definition 2.8]{quad128}). They first characterize the equivalence classes of these caps in terms of their possible extended basis types. For each, they construct a template of the cap, consisting of a basis $\{\vv{a}_1, \ldots, \vv{a}_8\}$ for the cap together with dependent elements $\vv{x}_1, \ldots, \vv{x}_r$, together with expressions for each $\vv{x}_i$ as sums of basis elements. They then determine which extended cap basis types correspond to equivalent caps, and which do not. In this section, we will consider each cap template from \cite{quad128} and use the Venn diagram method to determine which are equivalent and which are not.

\subsection{10-caps of dimension 7}

Table~\ref{dim7-10cap-templates} below is reproduced from \cite[Section 5]{quad128}, which gives three templates that cover all possibilities for caps of size $10$ and dimension $7$.

\begin{table}
    \centering
    \caption{Dependent set templates for $7$-dimensional $10$-caps with basis $\{\vv{a}_1, \ldots, \vv{a}_8\}$.}
    \label{dim7-10cap-templates}
    \begin{tabular}{|ll|}
        \hline
        \textbf{basis extended type} & \textbf{dependent set template} \\
        \hline
        $7 \hyph 5 \hyph (4)$
            & $\vv{x}_1 = \vv{a}_1 + \cdots + \vv{a}_7$ \\
            & $\vv{x}_2 = \vv{a}_4 + \cdots + \vv{a}_8$ \\
        \hline
        $5 \hyph 5 \hyph (2)$
            & $\vv{x}_1 = \vv{a}_1 + \cdots + \vv{a}_5$ \\
            & $\vv{x}_2 = \vv{a}_4 + \cdots + \vv{a}_8$ \\
        \hline
        $5 \hyph 5 \hyph (3)$
            & $\vv{x}_1 = \vv{a}_1 + \cdots + \vv{a}_5$ \\
            & $\vv{x}_2 = \vv{a}_3 + \cdots + \vv{a}_7$ \\
        \hline
    \end{tabular}
\end{table}

Following the same process as in Example~\ref{ex:venn-count-example}, we obtain Venn cardinality diagrams for each of these cap templates, displayed in Figure~\ref{dim7-10caps-Venn-counts}.

\begin{figure}
    \centering
    \begin{tabular}{|c|}
        \hline
        $7 \hyph 5 \hyph (4)$ \\
        \hline
        \begin{tikzpicture}[set/.style={fill=cyan,fill opacity=0.1}, scale = 0.3]
        \def\firstcircle{(0:1cm) circle (3cm)}
        \def\secondcircle{(180:1cm) circle (3cm)}
        
        \draw[set] \firstcircle;
        \draw[set] \secondcircle;

        \draw node at (4.4cm,3.4cm) {$X_2$};    
        \draw node at (-4.4cm,3.4cm) {$X_1$};    
        
            \draw node at (0:0cm) {$4$};
            \draw node at (180:2.7cm) {$4$};
            \draw node at (0:2.7cm) {$2$};
    \end{tikzpicture} \\
        \hline
    \end{tabular}
    \begin{tabular}{|c|}
        \hline
        $5 \hyph 5 \hyph (2)$ \\
        \hline
        \begin{tikzpicture}[set/.style={fill=cyan,fill opacity=0.1}, scale = 0.3]
        \def\firstcircle{(0:1cm) circle (3cm)}
        \def\secondcircle{(180:1cm) circle (3cm)}
        
        \draw[set] \firstcircle;
        \draw[set] \secondcircle;

        \draw node at (4.4cm,3.4cm) {$X_2$};    
        \draw node at (-4.4cm,3.4cm) {$X_1$};    
        
            \draw node at (0:0cm) {$2$};
            \draw node at (180:2.7cm) {$4$};
            \draw node at (0:2.7cm) {$4$};
    \end{tikzpicture} \\
        \hline
    \end{tabular}
    \begin{tabular}{|c|}
        \hline
        $5 \hyph 5 \hyph (3)$ \\
        \hline
        \begin{tikzpicture}[set/.style={fill=cyan,fill opacity=0.1}, scale = 0.3]
        \def\firstcircle{(0:1cm) circle (3cm)}
        \def\secondcircle{(180:1cm) circle (3cm)}
        
        \draw[set] \firstcircle;
        \draw[set] \secondcircle;

        \draw node at (4.4cm,3.4cm) {$X_2$};    
        \draw node at (-4.4cm,3.4cm) {$X_1$};    
        
            \draw node at (90:3.5cm) {$1$};    
            \draw node at (0:0cm) {$3$};
            \draw node at (180:2.7cm) {$3$};
            \draw node at (0:2.7cm) {$3$};
    \end{tikzpicture} \\
        \hline
    \end{tabular}

    \caption{The Venn cardinality diagrams for three caps of size $10$ and dimension $7$.}
    \label{dim7-10caps-Venn-counts}
\end{figure}
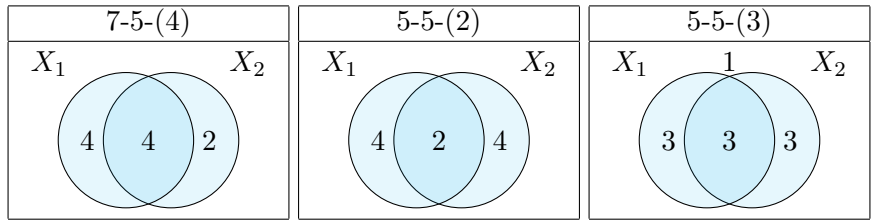

By Theorem~\ref{thm:diff-3}, it follows from Figure~\ref{dim7-10caps-Venn-counts} that caps of types $7 \hyph 5 \hyph (4)$ and $5 \hyph 5 \hyph (2)$ are equivalent, but caps of type $5 \hyph 5 \hyph 3$ are a different equivalent class. This is the same result as Corollary 5.5 of \cite{quad128}.

\subsection{11-caps of dimension 7}

Table~\ref{dim7-11cap-templates} below is reproduced from \cite[Section 6]{quad128}, which gives three templates that cover all possibilities for caps of size $11$ and dimension $7$.

\begin{table}
    \centering
    \caption{Dependent set templates for $7$-dimensional $11$-caps with basis $\{a_1, \ldots, a_8\}$.} \label{dim7-11cap-templates}
    \begin{tabular}{|rl|}
        \hline
        \textbf{basis extended type} & \textbf{dependent set template} \\
        \hline
        $7 \hyph 5 \hyph 5 \hyph (4,4,3)$
            & $x_1 = a_1 + \cdots + a_7$ \\
            & $x_2 = a_4 + \cdots + a_8$ \\
            & $x_3 = a_1 + a_2 + a_6 + a_7 + a_8$ \\
        \hline
        $5 \hyph 5 \hyph 5 \hyph (3,3,3)$
            & $x_1 = a_1 + \cdots + a_5$ \\
            & $x_2 = a_3 + \cdots + a_7$ \\
            & $x_3 = a_1 + a_3 + a_4 + a_6 + a_8$ \\
        \hline
        $5 \hyph 5 \hyph 5 \hyph (3,3,2)$
            & $x_1 = a_1 + \cdots + a_5$ \\
            & $x_2 = a_3 + \cdots + a_7$ \\
            & $x_3 = a_1 + a_2 + a_3 + a_7 + a_8$ \\
        \hline
    \end{tabular}
\end{table}

Figure~\ref{dim7-11caps-Venn-counts} shows the Venn cardinality diagrams for the three templates in Table~\ref{dim7-10cap-templates}, and this information is also collected in Table~\ref{dim7-11caps-Venn-counts-table}. Note that the multiset of Venn cardinalities is the same for all three templates, so they may all represent equivalent caps. 

\begin{figure}
    \centering
    \begin{tabular}{|c|}
        \hline
        $7 \hyph 5 \hyph 5 \hyph (4,4,3)$ \\
        \hline
        \begin{tikzpicture}[set/.style={fill=cyan,fill opacity=0.1}, scale = 0.4]
            \def\firstcircle{(90:2cm) circle (3cm)}
            \def\secondcircle{(210:2cm) circle (3cm)}
            \def\thirdcircle{(330:2cm) circle (3cm)}

            \draw[set] \firstcircle;
            \draw[set] \secondcircle;
            \draw[set] \thirdcircle;

            \draw node at (90:5.6cm) {$X_1$};    
            \draw node at (210:5.6cm) {$X_2$};    
            \draw node at (330:5.6cm) {$X_3$};    
        
                \draw node at (90:2.7cm) {$2$};    
                \draw node at (210:2.7cm) {$1$};
                \draw node at (330:2.7cm) {$1$};
                \draw node at (0,0) {$2$};
                \draw node at (270:2cm) {$1$};
                \draw node at (30:2cm) {$2$};
                \draw node at (150:2cm) {$2$};
    \end{tikzpicture} \\
        \hline
    \end{tabular}
    \begin{tabular}{|c|}
        \hline
        $5 \hyph 5 \hyph 5 \hyph (3,3,3)$ \\
        \hline
        \begin{tikzpicture}[set/.style={fill=cyan,fill opacity=0.1}, scale = 0.4]
            \def\firstcircle{(90:2cm) circle (3cm)}
            \def\secondcircle{(210:2cm) circle (3cm)}
            \def\thirdcircle{(330:2cm) circle (3cm)}

            \draw[set] \firstcircle;
            \draw[set] \secondcircle;
            \draw[set] \thirdcircle;

            \draw node at (90:5.6cm) {$X_1$};    
            \draw node at (210:5.6cm) {$X_2$};    
            \draw node at (330:5.6cm) {$X_3$};    
        
                \draw node at (90:2.7cm) {$2$};    
                \draw node at (210:2.7cm) {$2$};
                \draw node at (330:2.7cm) {$2$};
                \draw node at (0,0) {$2$};
                \draw node at (270:2cm) {$1$};
                \draw node at (30:2cm) {$1$};
                \draw node at (150:2cm) {$1$};
    \end{tikzpicture} \\
        \hline
    \end{tabular}
    \begin{tabular}{|c|}
        \hline
        $5 \hyph 5 \hyph 5 \hyph (3,3,2)$ \\
        \hline
        \begin{tikzpicture}[set/.style={fill=cyan,fill opacity=0.1}, scale = 0.4]
            \def\firstcircle{(90:2cm) circle (3cm)}
            \def\secondcircle{(210:2cm) circle (3cm)}
            \def\thirdcircle{(330:2cm) circle (3cm)}

            \draw[set] \firstcircle;
            \draw[set] \secondcircle;
            \draw[set] \thirdcircle;

            \draw node at (90:5.6cm) {$X_1$};    
            \draw node at (210:5.6cm) {$X_2$};    
            \draw node at (330:5.6cm) {$X_3$};    
        
                \draw node at (90:2.7cm) {$1$};    
                \draw node at (210:2.7cm) {$2$};
                \draw node at (330:2.7cm) {$2$};
                \draw node at (0,0) {$1$};
                \draw node at (270:2cm) {$1$};
                \draw node at (30:2cm) {$2$};
                \draw node at (150:2cm) {$2$};
    \end{tikzpicture} \\
        \hline
    \end{tabular}

    \caption{The Venn cardinality diagrams for three caps of size $11$ and dimension $7$.}
    \label{dim7-11caps-Venn-counts}
\end{figure}
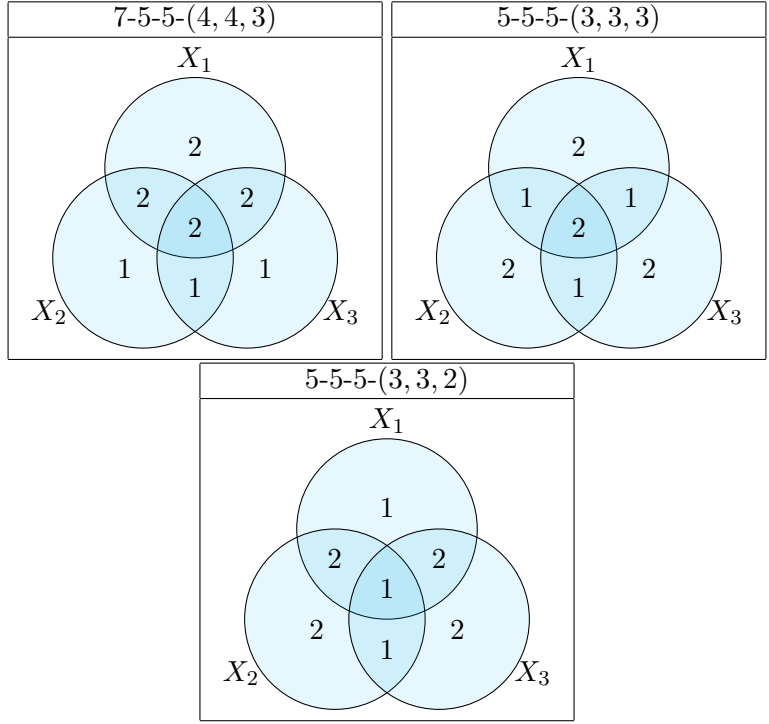

\begin{table}
    \centering
    \caption{Venn region counts $\abs{v(\vv{a})}$ for the three templates for $7$-dimensional $11$-caps from Table~\ref{dim7-11cap-templates}.}
    \label{dim7-11caps-Venn-counts-table}
    \[
    \begin{array}{|c|ccccc|}
        \hline
        & \textbf{7-5-5-(4,4,3)} && \textbf{5-5-5-(3,3,3)} && \textbf{5-5-5-(3,3,2)} \\ 
        \vv{a} & \abs{v(\vv{a})} && \abs{v(\vv{a})} && \abs{v(\vv{a})} \\ \hline
        000 & 0 && 0 && 0 \\
        100 & 2 && 2 && 1 \\
        010 & 1 && 2 && 2 \\
        001 & 1 && 2 && 2 \\
        110 & 2 && 1 && 2 \\
        011 & 1 && 1 && 1 \\
        101 & 2 && 1 && 2 \\
        111 & 2 && 2 && 1 \\
        \hline
    \end{array}
    \]
\end{table}



Let
\[
    A = \begin{bmatrix} 1 & 1 & 0 \\ 0 & 1 & 1 \\ 0 & 0 & 1 \end{bmatrix}
    \qquad \text{and} \qquad
    B = \begin{bmatrix} 1 & 0 & 0 \\ 1 & 1 & 0 \\ 0 & 0 & 1 \end{bmatrix}.
\]
It is straightforward to check that $A$ permutes the Venn regions of a $7 \hyph 5 \hyph 5 \hyph (4,4,3)$ template cap to the $5 \hyph 5 \hyph 5 \hyph (3,3,3)$ cap appropriately, and that $B$ permutes those of a $5 \hyph 5 \hyph 5 \hyph (3,3,3)$ template cap to those of a $5 \hyph 5 \hyph 5 \hyph (3,3,2)$ cap. Since these templates cover all possible $11$-caps of dimension $7$, we conclude that all such caps are equivalent. This is the same result as Corollary 6.9 of \cite{quad128}.

\subsection{12-caps of dimension 7}

Table~\ref{dim7-12cap-templates} below is reproduced from \cite[Section 7]{quad128}, which gives three templates that cover all possibilities for caps of size $12$ and dimension $7$.

\begin{table}[h]
    \centering
    \caption{Dependent set templates for $7$-dimensional $12$-caps with basis $\{a_1, \ldots, a_8\}$.} \label{dim7-12cap-templates}
    \begin{tabular}{|rl|}
        \hline
        \textbf{basis extended type} & \textbf{dependent set template} \\
        \hline
        $7 \hyph 5 \hyph 5 \hyph 5 \hyph (4,4,4,3,3,3)$
            & $x_1 = a_1 + \cdots + a_7$ \\
            & $x_2 = a_4 + \cdots + a_8$ \\
            & $x_3 = a_1 + a_2 + a_6 + a_7 + a_8$ \\
            & $x_4 = a_1 + a_3 + a_5 + a_7 + a_8$ \\
        \hline
        $5 \hyph 5 \hyph 5 \hyph 5 \hyph (2,3,3,3,3,3)$
            & $x_1 = a_1 + \cdots + a_5$ \\
            & $x_2 = a_1 + a_2 + a_3 + a_7 + a_8$ \\
            & $x_3 = a_3 + \cdots + a_7$ \\
            & $x_4 = a_2 + a_3 + a_4 + a_6 + a_8$ \\
        \hline
        $5 \hyph 5 \hyph 5 \hyph 5 \hyph (2,3,3,3,3,2)$
            & $x_1 = a_1 + \cdots + a_5$ \\
            & $x_2 = a_1 + a_2 + a_3 + a_7 + a_8$ \\
            & $x_3 = a_3 + \cdots + a_7$ \\
            & $x_4 = a_2 + a_4 + a_6 + a_7 + a_8$ \\
        \hline
    \end{tabular}
\end{table}

Table~\ref{dim7-12caps-Venn-counts-table} shows the Venn region cardinality information diagrams for the three templates in Table~\ref{dim7-12cap-templates}. Again, the multiset of Venn cardinalities is the same for all three templates, so they may all represent equivalent caps. 

\begin{table}
    \centering
    \caption{Venn region counts $\abs{v(\vv{a})}$ for the three templates for $7$-dimensional $12$-caps from Table~\ref{dim7-12cap-templates}.} \label{dim7-12caps-Venn-counts-table}
    \[
    \begin{array}{|c|ccccc|}
        \hline
        & \textbf{7-5-5-5} && \textbf{5-5-5-5} && \textbf{5-5-5-5} \\ 
        & \textbf{-(4,4,4,3,3,3)} && \textbf{-(2,3,3,3,3,3)} && \textbf{-(2,3,3,3,3,2)} \\ 
        
        \vv{a} & \abs{v(\vv{a})} && \abs{v(\vv{a})} && \abs{v(\vv{a})} \\ \hline
        0000 & 0 && 0 && 0 \\
        1000 & 1 && 1 && 1 \\
        0100 & 1 && 1 && 1 \\
        0010 & 1 && 1 && 1 \\
        0001 & 1 && 1 && 1 \\
        1100 & 1 && 1 && 1 \\
        1010 & 1 && 1 && 1 \\
        1001 & 1 && 0 && 0 \\
        0110 & 0 && 1 && 0 \\
        0101 & 0 && 1 && 1 \\
        0011 & 0 && 1 && 1 \\
        1110 & 1 && 0 && 1 \\
        1101 & 1 && 1 && 1 \\
        1011 & 1 && 1 && 1 \\
        0111 & 1 && 0 && 1 \\
        1111 & 1 && 1 && 0 \\
        \hline
    \end{array}
    \]
\end{table}

Let
\[
    A = \begin{bmatrix} 1 & 0 & 1 & 1 \\ 0 & 1 & 1 & 0 \\ 0 & 0 & 0 & 1 \\ 0 & 0 & 1 & 0 \end{bmatrix}
    \qquad \text{and} \qquad
    B = \begin{bmatrix} 1 & 0 & 1 & 0 \\ 0 & 1 & 0 & 0 \\ 0 & 0 & 1 & 0 \\ 0 & 0 & 0 & 1 \end{bmatrix}.
\]
It is straightforward to check that $A$ permutes the Venn regions of a $7 \hyph 5 \hyph 5 \hyph 5 \hyph (4,4,4,3,3,3)$ template cap to the $5 \hyph 5 \hyph 5 \hyph 5 \hyph (2,3,3,3,3,3)$ cap appropriately, and that $B$ permutes those of a $5 \hyph 5 \hyph 5 \hyph 5 \hyph (2,3,3,3,3,3)$ template cap to those of a $5 \hyph 5 \hyph 5 \hyph 5 \hyph (2,3,3,3,3,2)$ cap. Since these templates cover all possible $12$-caps of dimension $7$, we conclude that all such caps are equivalent. This is the same result as Corollary 7.8 of \cite{quad128}.


\printbibliography

\end{document}